\documentclass[a4paper, 11pt]{article}

\usepackage[utf8]{inputenc}
\usepackage[T1]{fontenc}

\usepackage{amsthm}
\usepackage{amssymb}
\usepackage{mathtools}
\usepackage{stmaryrd}	\SetSymbolFont{stmry}{bold}{U}{stmry}{m}{n}	% To avoid warnings
\usepackage{bm}
\usepackage{mathrsfs}
\usepackage{tikz-cd}
\usetikzlibrary{positioning, shapes.geometric, patterns, graphs, decorations.pathmorphing}

\usepackage{paralist}

\usepackage[colorlinks]{hyperref}
\hypersetup{
	linkcolor=blue,
	citecolor=teal!50!green,
	urlcolor=cyan,
	pdftitle = {Generalized zeta integrals on certain real prehomogeneous vector spaces},
	pdfauthor = {Wen-Wei Li}
}

\newcommand{\Z}{\ensuremath{\mathbb{Z}}}
\newcommand{\Q}{\ensuremath{\mathbb{Q}}}
\newcommand{\R}{\ensuremath{\mathbb{R}}}
\newcommand{\CC}{\ensuremath{\mathbb{C}}}

\newcommand{\Tr}{\operatorname{tr}}
\newcommand{\topwedge}{\ensuremath{\bigwedge^{\mathrm{max}}}}

\newcommand{\Gal}{\operatorname{Gal}}

\newcommand{\dd}{\mathop{}\!\mathrm{d}}
\newcommand{\Schw}{\ensuremath{\mathcal{S}}}	% Schwartz space
\newcommand{\relg}[1]{\ensuremath{\underset{#1}{>}}}	% greater relative to...
\newcommand{\relgeq}[1]{\ensuremath{\underset{#1}{\geq}}}	% greater or equal relative to...
\newcommand{\relgg}[1]{\ensuremath{\underset{#1}{\gg}}}		% much greater relative to...

\newcommand{\lrangle}[1]{\ensuremath{\langle #1 \rangle}}

\newcommand{\identity}{\ensuremath{\mathrm{id}}}

\newcommand{\Hom}{\operatorname{Hom}}

\newcommand{\End}{\operatorname{End}}
\newcommand{\rightiso}{\ensuremath{\stackrel{\sim}{\rightarrow}}}

\newcommand{\Ker}{\operatorname{ker}}

\newcommand{\dotimes}[1]{\ensuremath{\underset{#1}{\otimes}}}

\newcommand{\Lie}{\operatorname{Lie}}

\newcommand{\Spec}{\operatorname{Spec}}
\newcommand{\Gm}{\ensuremath{\mathbb{G}_\mathrm{m}}}

\newcommand{\Res}{\operatorname{Res}}
\newcommand{\utimes}[1]{\ensuremath{\overset{#1}{\times}}}
\newcommand{\dtimes}[1]{\ensuremath{\underset{#1}{\times}}}
\newcommand{\Supp}{\operatorname{Supp}}
\newcommand{\GL}{\operatorname{GL}}

\theoremstyle{plain}
\newtheorem{proposition}{Proposition}
\newtheorem{lemma}[proposition]{Lemma}
\newtheorem{theorem}[proposition]{Theorem}
\newtheorem{corollary}[proposition]{Corollary}
\theoremstyle{definition}
\newtheorem{definition}[proposition]{Definition}
\newtheorem{definition-theorem}[proposition]{Definition-Theorem}
\newtheorem{definition-proposition}[proposition]{Definition-Proposition}
\newtheorem{remark}[proposition]{Remark}
\newtheorem{hypothesis}[proposition]{Hypothesis}

\newtheorem{example}[proposition]{Example}

\theoremstyle{definition}

\theoremstyle{plain}

\numberwithin{equation}{section}
\numberwithin{proposition}{section}
\numberwithin{conj}{section}	% In the Introduction only

\renewcommand{\Re}{\operatorname{Re}}	% Symbol for the real and imaginary parts

\renewcommand{\emptyset}{\ensuremath{\varnothing}}	% Symbol for the emptyset

\usepackage{amsmath}
\usepackage[nottoc]{tocbibind}

\usepackage{geometry}
\usepackage{lmodern}

\title{Generalized zeta integrals on certain real prehomogeneous vector spaces}
\author{Wen-Wei Li}
\date{}
\begin{document}

\maketitle

\begin{abstract}
	Let $X$ be a real prehomogeneous vector space under a reductive group $G$, such that $X$ is an absolutely spherical $G$-variety with affine open orbit. We define local zeta integrals that involve the integration of Schwartz--Bruhat functions on $X$ against generalized matrix coefficients of admissible representations of $G(\mathbb{R})$, twisted by complex powers of relative invariants. We establish the convergence of these integrals in some range, the meromorphic continuation as well as a functional equation in terms of abstract $\gamma$-factors. This subsumes the Archimedean zeta integrals of Godement--Jacquet, those of Sato--Shintani (in the spherical case), and the previous works of Bopp--Rubenthaler. The proof of functional equations is based on Knop's results on Capelli operators.
\end{abstract}

% Another Abstract:
%The Godement-Jacquet zeta integrals and Sato's prehomogeneous zeta integrals share a common feature: they both involve Schwartz functions and Fourier transforms on prehomogeneous vector spaces. In this talk I will sketch a common generalization in the local Archimedean case. Specifically, for a reductive prehomogeneous vector space which is also a spherical variety, I will define the zeta integrals of generalized matrix coefficients of admissible representations against Schwartz functions, prove their convergence and meromorphic continuation, and establish the local functional equation. Our arguments are based on various estimates on generalized matrix coefficients and Knop's work on invariant differential operators.

{\scriptsize
\begin{tabular}{ll}
	\textbf{MSC (2010)} & 11S40; 11S90 43A85 \\
	\textbf{Keywords} & Zeta integrals, prehomogeneous vector spaces, Capelli operators
\end{tabular}}

\tableofcontents

\section{Introduction}\label{sec:intro}

\subsection{Main results}
Prehomogeneous vector spaces are a rich source of zeta integrals with meromorphic continuation and functional equation. The aim of this article is to extend the scope of this construction by incorporating generalized matrix coefficients of admissible representations of a connected reductive group over $\R$. Let us begin by summarizing the main Theorems of this article. To put things into context, we will discuss their relation to existing theories in the next subsection.

A \emph{reductive prehomogeneous vector space} over $\R$ is a triplet $(G, \rho, X)$ where $G$ is a connected reductive $\R$-group, $X \neq \{0\}$ is a finite-dimensional $\R$-vector space, and $\rho: G \to \GL(X)$ is a homomorphism of algebraic groups such that $X$ has a Zariski-dense open $G$-orbit $X^+$. By convention, $\GL(X)$ acts on the right of $X$, thus acts on the left of various spaces of functions on $X$. Suppose furthermore that $\partial X := X \smallsetminus X^+$ is a hypersurface. Then $\partial X$ is defined by $f_1 \cdots f_r = 0$ where $f_1, \ldots, f_r \in \R[X]$ are irreducible polynomials, unique up to $\R^\times$, called the \emph{basic relative invariants} under the $G$-action. We refer to \S\ref{sec:relative-invariants} or \cite{Sa89,Ki03} for generalities about prehomogeneous vector spaces.

Recall that a homogeneous $G$-space $X^+$ is called \emph{spherical}, also known as \emph{absolutely spherical}, if there is an open Borel orbit in $X^+_{\CC} := X^+ \times_{\R} \CC$ under $G_{\CC}$-action. Let $\mathbf{X}^*(G) := \Hom_{\text{alg.\ grp.} / \R}(G, \Gm)$. Our assumptions (Hypothesis \ref{hyp:PVS}) are:
\begin{compactenum}[(i)]
	\item $X^+$ is a spherical homogeneous $G$-space;
	\item $\partial X$ is a hypersurface, defined by $f_1 \cdots f_r = 0$ where $f_1, \ldots, f_r$ are basic relative invariants, with eigencharacters $\omega_1, \ldots, \omega_r \in \mathbf{X}^*(G)$.
\end{compactenum}

The reductive prehomogeneous vector spaces satisfying only (i) are called \emph{multiplicity-free spaces}. For irreducible $\rho$, they have been classified by V.\ Kac \cite{Kac80}. The general classification is done independently in \cite{BR96, Le98}. See also \cite{Kn98}.

Let $\mathbf{X}^*_\rho(G) \subset \mathbf{X}^*(G)$ denote the subgroup eigencharacters of rational relative invariants in $\R(X)$. It is known that $\mathbf{X}^*_\rho(G) = \bigoplus_{i=1}^r \Z \omega_i$. For any commutative ring $A$, set $\Lambda_A := \mathbf{X}^*_\rho(G) \otimes_{\Z} A$. For $\lambda = \sum_{i=1}^r \omega_i \otimes \lambda_i \in \Lambda_{\R}$, we write $\Re(\lambda) \relgg{X} 0$ to indicate $\lambda_i \gg 0$ for all $i$. We also write
\[ |f|^\lambda(x) := \prod_{i=1}^r |f_i(x)|^{\lambda_i}, \quad x \in X(\R). \]

It is convenient to employ the language of \emph{half-densities} on real manifolds; see \S\ref{sec:densities}. They are $C^\infty$-sections of a canonical line bundle $\mathcal{L}^{1/2}$ over the manifold, and can be thought as square roots of measures. Locally they can be represented as $f |\omega|^{1/2}$ where $f$ is a $C^\infty$-function and $\omega$ is a differential form of top degree. For example, given $\Omega \in \topwedge \check{X} \smallsetminus \{0\}$, we have the half-density $|\Omega|^{1/2}$ on $X(\R)$; it is translation-invariant, and varies by $|\det\rho|^{1/2}$ under $G(\R)$-action. The product of two half-densities is a density, whose integration makes sense.

Let $C^\infty(X^+)$ denote the Fréchet space of $C^\infty$-smooth densities. Likewise, we have the Fréchet space of Schwartz--Bruhat half-densities $\Schw(X)$, which equals $\Schw_0(X)|\Omega|^{1/2}$ by choosing $\Omega$, where $\Schw_0(X)$ is the scalar-valued Schwartz--Bruhat space. They are both smooth $G(\R)$-representations.

It turns out that our assumptions on $(G, \rho, X)$ passes to its dual (i.e.\ contragredient) $(G, \check{\rho}, \check{X})$, and $\mathbf{X}^*_\rho(G) = \mathbf{X}^*_{\check{\rho}}(G)$. Upon choosing an additive character $\psi$ of $\R$, one can define the Fourier transform $\mathcal{F}: \Schw(X) \rightiso \Schw(\check{X})$ of half-densities, in such a way that $\mathcal{F}$ is $G(\R)$-equivariant.

Our zeta integrals are associated with admissible representations of $G(\R)$. The natural formalism is that of SAF representations (smooth, admissible of moderate growth, Fréchet --- see \cite{BK14}), also known as Casselman--Wallach representations. The category of SAF representations is equivalent to that of Harish-Chandra $(\mathfrak{g}, K)$-modules by taking $K$-finite parts. For each SAF representation $\pi$ realized on a Fréchet space $V_\pi$, the $\CC$-vector space
\[ \mathcal{N}_\pi(X^+) := \Hom_{G(\R)}\left( \pi, C^\infty(X^+) \right) \]
is known to be finite-dimensional where we take the continuous and $G(\R)$-equivariant $\Hom$-space. For each vector $v \in V_\pi$ and $\eta \in \mathcal{N}_\pi(X^+)$, we call $\eta(v) \in \mathcal{N}_\pi(X^+)$ a \emph{generalized matrix coefficient} of $\pi$. It can be reduced to the usual scalar-valued generalized matrix coefficients by trivializing $\mathcal{L}^{1/2}$ on $X^+(\R)$ equivariantly (Lemma \ref{prop:avoid-density}).

The \emph{generalized zeta integrals} in question are
\[ Z_\lambda\left( \eta, v, \xi \right) := \int_{X^+(\R)} \eta(v) |f|^\lambda \xi, \]
where $\eta \in \mathcal{N}_\pi(X^+)$, $v \in V_\pi$, and $\xi \in \Schw(X)$. The goal of this article is to prove three basic properties of these integrals, in increasing level of difficulty:

\begin{description}
	\item[Convergence (Theorem \ref{prop:convergence})] The integral $Z_\lambda(\eta, v, \xi)$ converges for $\Re(\lambda) \relgeq{X} \kappa$ for some $\kappa \in \Lambda_{\R}$ depending only on $\pi$ and $(G, \rho, X)$, and it is jointly continuous in $(v, \xi)$ in that range.
	\item[Meromorphic continuation (Theorem \ref{prop:meromorphy})] $Z_\lambda(\eta, v, \xi)$ admits a meromorphic continuation to all $\lambda \in \Lambda_{\CC}$. To be precise, there exists a holomorphic function $L(\eta, \lambda)$ on $\Lambda_{\CC}$ for any given $\eta$, not identically zero, such that $LZ_\lambda(\eta, v, \xi) := L(\eta, \lambda) Z_\lambda(\eta, v, \xi)$ extends holomorphically to all $\lambda \in \Lambda_{\CC}$.
	\item[Functional equation (Theorem \ref{prop:LFE})] Fix an additive character $\psi$ and denote the integral for $(G, \check{\rho}, \check{X})$ as $\check{Z}_\lambda$. There is then a unique meromorphic family of $\CC$-linear maps $\gamma(\pi, \lambda): \mathcal{N}_\pi(\check{X}^+) \to \mathcal{N}_\pi(X^+)$, called the $\gamma$-factor, such that
	\[ \check{Z}_\lambda \left( \check{\eta}, v, \mathcal{F}\xi \right) = Z_\lambda\left( \gamma(\lambda, \pi)(\check{\eta}), v, \xi \right) \]
	for all $\check{\eta} \in \mathcal{N}_\pi(\check{X}^+)$, $v \in V_\pi$ and $\xi \in \Schw(X)$, where both sides are viewed as meromorphic families in $\lambda \in \Lambda_{\CC}$.
\end{description}

Moreover, one can obtain slightly more information on the ``denominator'' $L(\eta, \lambda)$, and describe the dependence of $\gamma(\lambda, \pi)$ on $\psi$; it turns out that the $\gamma$-factor, which is actually a linear transform, is generically invertible (Proposition \ref{prop:gamma-symmetry}). We refer to the cited Theorems for the precise statements.

Note that our formalism is non-trivial only when $\mathcal{N}_\pi(X^+) \neq \{0\}$; in other words, $\pi$ must be \emph{distinguished} by $X^+(\R)$. Distinguished representations and their generalized matrix coefficients are the main concerns of harmonic analysis on spherical varieties.

The same result hold for prehomogeneous vector spaces over $\CC$; see \S\ref{sec:cplx-case}.

\subsection{Background}
The prototype of zeta integrals in representation theory is Tate's thesis. His idea is to study the $L$-factors by integrating Schwartz--Bruhat functions against characters by embedding $F^\times$ in $F$, and then interpret the functional equation in terms of Fourier transform. There are at least two well-known extensions of Tate's theory, both fitting into our general scenario. We discuss only the local case $F = \R$.

\begin{enumerate}
	\item \emph{Godement--Jacquet theory} (Example \ref{eg:GJ}). Let $D$ be a central simple $\R$-algebra with $\dim D = n^2$, and let $D^\times \times D^\times$ act on the right of $X := D$ by $x\rho(g,h) = h^{-1}xg$. This gives a reductive prehomogeneous vector space $(D^\times \times D^\times, \rho, D)$ with open orbit $X^+ := D^\times$, which is spherical. The irreducible SAF representations $\pi$ of $D^\times(\R) \times D^\times(\R)$ with $\mathcal{N}_\pi(D^\times) \neq \{0\}$ are of the form $\sigma \boxtimes \check{\sigma}$, where $\check{\sigma}$ is the contragredient representation of $\sigma$. In this case, $\mathcal{N}_\pi(D^\times)$ is spanned by the \emph{matrix coefficient} map
	\[ v \otimes \check{v} \mapsto \lrangle{\check{v}, \sigma(\cdot) v} \cdot |\mathrm{Nrd}|^{-n/2} |\Omega|^{1/2} \]
	where $\mathrm{Nrd}$ is the reduced norm on $D$, and $\Omega$ is a volume form as before. Note that $\mathrm{Nrd} \in \R[D]$ is the basic relative invariant.

	Let $v \otimes \check{v} \in V_\sigma \otimes V_{\check{\sigma}}$. The Godement--Jacquet zeta integral in this setting is
	\[ Z^{\mathrm{GJ}}\left( \lambda, v \otimes \check{v}, \xi_0 \right) := \int_{D^\times(\R)} \lrangle{\check{v}, \pi(x)v} \left| \mathrm{Nrd}(x) \right|^{\lambda + \frac{n-1}{2}} \xi_0(x) \dd^\times x \]
	where $\dd^\times x := |\mathrm{Nrd}|^{-n/2} |\Omega|^{1/2}$ is a Haar measure on $D^\times(\R)$, and $\xi_0$ is any Schwartz--Bruhat function on $D(\R) \simeq \R^{n^2}$. It is routine to check that $Z^{\mathrm{GJ}}(\lambda + \frac{1}{2}, v \otimes \check{v}, \xi_0 )$ equals the generalized zeta integral introduced previously. Moreover, by relating $\check{X}^+$ to $X^+$ appropriately, we recover the Godement--Jacquet functional equation
	\[ Z^{\mathrm{GJ}}\left( 1-\lambda, \check{v} \otimes v, \mathcal{F}\xi_0 \right) = \gamma^{\mathrm{GJ}}(\lambda, \pi) Z^{\mathrm{GJ}}\left( \lambda, v \otimes \check{v}, \xi_0 \right), \]
	where the left hand side is defined with respect to $\check{\pi}$, and the self-dual Haar measure on $D(\R)$ is used. These integrals and their global avatar give rise to the standard $L$-factor $L(\lambda, \pi, \mathrm{Std})$, by taking the greatest common divisors over all $\xi_0$.
	
	\item \emph{Sato--Shintani theory} (Example \ref{eg:SS}). Consider a triplet $(G, \rho, X)$ as in our generalized setting, but take $\pi$ to be the trivial representation of $G(\R)$. Then $\mathcal{N}_\pi(X^+)$ is in bijection with the $G(\R)$-orbits $O_1, \ldots, O_m$ in $X^+(\R)$. The $G(\R)$-orbits on $\check{X}(\R)$ turn out to be in bijection with $O_1, \ldots, O_m$. The resulting zeta integral is, up to a shift in $\lambda$, the one considered by Sato--Shintani \cite{SS74} and completed by F.\ Sato \cite{Sa82} for the case in several variables, following the pioneering works of M.\ Sato on prehomogeneous vector spaces. They name the functional equation as the \emph{Fundamental Theorem}. The condition on sphericity of $X^+$ can be removed in this setting.
	
	Specifically, Sato and Shintani worked only in the global case; the local zeta integrals are introduced later by Igusa \textit{et al.} We refer to \cite{Sa89} for a more detailed survey for the local integrals, and to \cite{Sa94, Sai03} for the relation between local and global integrals. Note that the functional equation in the non-Archimedean case is known only under some assumptions on $(G, \rho, X)$; see \cite{Sa89}.
\end{enumerate}

In both theories the functional equation is the hardest part. We can make a further comparison as follows.
\begin{itemize}
	\item The Godement--Jacquet integrals are directly related to Langlands program since they yield standard $L$-factors; however the corresponding functional equation is proved by an \textit{ad hoc} argument, namely by reduction to Tate's thesis (see \cite{GH11-2}).

	\item On the other hand, Sato--Shintani functional equations are proved in \cite{SS74, Sa82} by a general, geometric reasoning. The corresponding $L$-factors, whenever they are identified, are highly degenerate; this is not surprising since the zeta integrals involve only twists of the trivial representation of $G(F)$. Their applicability to Langlands program is therefore limited, despite the flexibility of choosing $(G, \rho, X)$.
\end{itemize}

In \cite{LiLNM}, the author proposed a general framework to define zeta integrals whenever one has a spherical homogeneous $G$-space $X^+$, an equivariant embedding $X^+ \hookrightarrow X$ together with a reasonable notion of Schwartz space and Fourier transform. That project is largely speculative, the only accessible case being the setting of prehomogeneous vector spaces mentioned above. The belief behind \cite{LiLNM} is that the three basic properties of such zeta integrals over a local field $F$, namely: convergence, meromorphic continuation and functional equation, should have a uniform proof based on general principles. Moreover, we expect some global applications to the study of periods or sums (possibly infinite) of $L$-values, although this is surely a long-term goal. In this connection, we remark that Sakellaridis \cite{Sak12} made unramified computations for non-exceptional groups in Kac's classification and concluded that they give only ``known'' $L$-factors.

Generalized zeta integrals over $\R$ have also been studied in \cite{BR05} for a specific class of triplets $(G, \rho, X)$ and representations $\pi$. In particular, they obtained the functional equation via explicit computations, and obtained a more precise description of $\mathcal{N}_\pi(X^+)$ and $\gamma(\lambda, \pi)$. Another generalization in this direction is due to F.\ Sato \cite{Sa94, Sa06}, which puts more emphasis on the global picture involving periods of automorphic forms and allows some non-spherical cases. We hope to explore the possible extensions of our theory to his cases in the future.

When the local field $F$ is $p$-adic, $G$ is split and $X^+$ satisfies the \emph{wavefront} condition, some positive results about generalized zeta integrals have been obtained in \cite[Chapter 6]{LiLNM}, including a functional equation under extra assumptions on $\partial X$.

For the Archimedean case, say $F=\R$, the convergence and meromorphic continuation have been obtained in \cite{Li18} when $X^+$ is a finite cover of an algebraic symmetric space under $G$. Many of the arguments therein are general, requiring only some expected properties of generalized matrix coefficients as input. This article completes the Archimedean case in full generality.

\subsection{About the proofs}
Fix a maximal compact subgroup $K \subset G(\R)$ and a $G(\R)$-equivariant trivialization of the line bundle $\mathcal{L}^{1/2}$ on $X^+(\R)$, as in Lemma \ref{prop:avoid-density}. For every $\eta \in \mathcal{N}_\pi(X^+)$, we denote by $\eta_0$ the corresponding morphism from $\pi$ to $C^\infty(X^+(\R); \CC)$, the space of scalar-valued $C^\infty$ functions.

The convergence for $\Re(\lambda) \relgg{X} 0$ is the first and the simplest step. \textit{Grosso modo}, it suffices to show that $\eta_0(v)$ is of \emph{moderate growth} on $X^+(\R)$, uniformly in $v$; see Proposition \ref{prop:soft-bound}. The argument has been sketched in \cite[\S 6.6]{Li18}, but the proof of moderate growth therein for essentially symmetric spaces is unnecessarily complicated. By using available estimates for generalized matrix coefficients, for example those in \cite{KKS17} or \cite{Li19}, we are able to prove the general case here. Furthermore, we show that $(v, \xi) \mapsto Z_\lambda(\eta, v, \xi)$ is jointly continuous and bounded in vertical strips in the range of convergence.

The meromorphic continuation of $Z_\lambda$ is achieved by the machinery of Bernstein--Sato $b$-functions. The idea based on differential operators is explained in \cite[\S 6.8]{Li18} which is in turn modeled on \cite{BD92}; the technique has also been employed for Sato--Shintani zeta integrals. The main input is the fact that for each $v \in V_\pi^{K\text{-fini}}$, the $\mathscr{D}_{X^+_{\CC}}$-module generated by $\eta_0(v)$ is \emph{holonomic}, where $\mathscr{D}_{X^+_{\CC}}$ is the sheaf of algebraic differential operators on $X^+_{\CC}$. In fact, we will show that there is a holomorphic function $L(\eta, \lambda)$ in $\lambda \in \Lambda_{\CC}$, which can be taken to be a product of inverses of $\Gamma$-functions, such that
\[ LZ_\lambda(\eta, v, \xi) := L(\eta, \lambda) Z_\lambda(\eta, v, \xi) \]
extends holomorphically to all $\lambda \in \Lambda_{\CC}$. Moreover, $LZ_\lambda(\eta, \cdot, \cdot)$ extends to a jointly continuous bilinear form on $V_\pi \times \Schw(X)$.

The required holonomicity is furnished by \cite{Li19}, and one can also deduce it from the arguments in \cite{AGM16}. The remaining arguments are the same as in \cite{Li18}.

The hardcore is the functional equation. We proceed in two steps.
\begin{enumerate}
	\item First, we produce a uniquely determined meromorphic family of linear maps
	\[ \gamma(\pi, \lambda): \mathcal{N}_\pi(\check{X}^+) \to \mathcal{N}_\pi(X^+), \quad \lambda \in \Lambda_{\CC} \]
	verifying the \emph{weak functional equation}
	\[ \check{Z}_\lambda \left( \check{\eta}, v, \mathcal{F}\xi \right) = Z_\lambda\left( \gamma(\lambda, \pi)(\check{\eta}), v, \xi \right), \quad \xi \in C^\infty_c(X^+), \; v \in V_\pi. \]
	Surely, here $C^\infty_c(X^+)$ is valued in half-densities. The idea is simple: given $v \in V_\pi^{K\text{-fini}}$, regard $\xi \mapsto L\check{Z}_\lambda(\check{\eta}, v, \mathcal{F}\xi)$ as a tempered distribution $T_\lambda(v)$ on $X(\R)$. One shows that it is $\mathcal{Z}(\mathfrak{g})$-finite and $K$-finite as $v$ is, and deduce it is $C^\infty$ on $X^+(\R)$ by the elliptic regularity theorem. Next, one shows that $v \mapsto T_\lambda(v)|f|^{-\lambda}$ extends to a holomorphic family in $\mathcal{N}_\pi(X^+)$; this yields the $\gamma$-factor after dividing by the ``denominator'' $L(\check{\eta}, \lambda)$.
	
	This step involves some finiteness properties, as well as an automatic continuity property for $v \mapsto T_\lambda(v)|f|^{-\lambda}$. For this purpose, we invoke some results from \cite{Li19}, although these ingredients have probably been known elsewhere. Another ingredient is the decomposition of $X^+(\R)$ in Proposition \ref{prop:rho-decomposition} and the accompanying Proposition \ref{prop:N-restriction}, which enter in the proof of holomorphy of $T_\lambda(\cdot)|f|^{-\lambda}$.
	
	\item Secondly, let $\lambda$ vary in a bounded open subset $U \subset \Lambda_{\CC}$. Observe that $\Delta_\lambda(\check{\eta}, v, \xi) := \check{Z}_\lambda \left( \check{\eta}, v, \mathcal{F}\xi \right) - Z_\lambda\left( \gamma(\lambda, \pi)(\check{\eta}), v, \xi \right)$ satisfies
	\[ \Delta_\lambda(\check{\eta}, v, h^M \xi) = 0, \quad \xi \in \Schw(X), \]
	where $h \in \R[X]$ is an appropriate relative invariant with zero locus $\partial X$ and $M \gg 0$. Using the uniqueness of $\gamma$-factors in the weak functional equation, we transform this equality into $\Delta_{\lambda - M\theta}\left( {}^\lambda \check{\eta}, v, \xi \right) = 0$ for all $\xi \in \Schw(X)$, where
	\begin{compactitem}
		\item $\theta$ is the eigencharacter of $h$,
		\item $\check{\eta} \mapsto {}^\lambda \check{\eta}$ is a holomorphic family of endomorphisms of $\mathcal{N}_\pi(\check{X}^+)$, given by the action of some (analytic) twists of a $G$-invariant algebraic differential operator on $\check{X}$, called the \emph{Capelli operator}.
	\end{compactitem}

	The functional equation will follow once $\check{\eta} \mapsto {}^\lambda \check{\eta}$ is shown to be generically invertible. We do this by first decomposing $\mathcal{N}_\pi(\check{X}^+)$ into generalized eigenspaces under $\mathcal{D}(\check{X}^+_{\CC})^{G_{\CC}}$, the algebra of invariant algebraic differential operators on $\check{X}^+_{\CC}$. Then we analyze the eigenvalues of the twists of Capelli operator via Knop's Harish-Chandra isomorphism \cite{Kn94, Kn98}. Eventually, the generic invertibility results from Knop's formula in \cite{Kn98} for the leading term. Here we make crucial use of the existence of non-degenerate relative invariants of our prehomogeneous vector spaces.
\end{enumerate}

The arguments above are completely disjoint from the Godement--Jacquet case. When $\pi$ is the trivial representation, it reduces to the proof of Sato--Shintani and F.\ Sato, in which case the effect of Capelli operator can be made explicit.

We also remark that the sphericity of $X^+_{\CC}$ is necessary for both the proofs of meromorphy and functional equation. In contrast, the convergence holds when $X^+$ is just \emph{real spherical}, i.e.\ when there is an open $P_0$-orbit in $X^+$ where $P_0 \subset G$ is a minimal parabolic subgroup.

\subsection{Organization of this article}
The general conventions are presented in \S\ref{sec:conventions}.

In \S\ref{sec:PVS}, we introduce the basic notions about prehomogeneous vector spaces, density bundles, the action of differential operators on half-densities, and the Fourier transform for both the scalar and half-density cases. In particular, we enunciate the Hypothesis \ref{hyp:PVS} about the prehomogeneous vector spaces.

In \S\ref{sec:desiderata}, we define the generalized matrix coefficients of an SAF representation of $G(\R)$, define the generalized zeta integrals $Z_\lambda(\eta, v, \xi)$ and state the main Theorems \ref{prop:convergence}, \ref{prop:meromorphy}, \ref{prop:LFE}. Granting these results, we also describe the inverse of $\gamma$-factor and its dependence on $\psi$ in Proposition \ref{prop:gamma-symmetry}. Note that there is a self-dual version of Fourier transform and $\gamma$-factors, which are more natural in some circumstances, and will be discussed in Remark \ref{rem:sd-version}.

In \S\ref{sec:examples}, the zeta integrals of Godement--Jacquet and Sato--Shintani (in the local, spherical case), together with their functional equations, are shown to be special cases of our formalism. In \S\ref{sec:cplx-case}, we state the complex case and reduce it to the real case by restriction of scalars.

The convergence of $Z_\lambda$ for $\Re(\lambda) \relgg{X} 0$ and the meromorphic continuation are proved in \S\ref{sec:conv-mero}. The section also records some auxiliary results for later use.

The intermezzo \S\ref{sec:diff-op} is mainly a recap of Knop's Harish-Chandra isomorphism for multiplicity-free spaces over $\CC$. In that section, we will also define the relevant Capelli operators and their twists, in both the algebraic and analytic setting. The upshot is the crucial computation of leading terms in Propositions \ref{prop:HC-twist}, \ref{prop:top-nonvanishing}.

The functional equation is established in \S\ref{sec:FE} by proving first a weak functional equation in \S\ref{sec:gamma}, which also determines the $\gamma$-factor. Then we deduce the full version by using Capelli operators and their twists.

\subsection{Acknowledgements}
The author is grateful to Miyu Suzuki and Satoshi Wakatsuki for discussions on the works \cite{Sa94, Sa06} of F.\ Sato. This work is supported by NSFC-11922101.

\subsection{Conventions}\label{sec:conventions}
\paragraph*{Fields}
Field extensions are written in the form $E|F$. The Galois group of a Galois extension $E|F$ will be denoted by $\Gal(E|F)$.

The additive characters of $\R$ are nontrivial continuous homomorphisms $\psi: \R \to \{z \in \CC^\times: |z|=1 \}$. The additive characters form an $\R^\times$-torsor under the action $\psi \xmapsto{a} [\psi_a: t \mapsto \psi(at)]$ where $a \in \R^\times$.

\paragraph*{Varieties and groups}
Let $F$ be a field. By an $F$-variety we mean an integral separated scheme of finite type over $\Spec F$. If $X$ is an $F$-variety and $E|F$ is a field extension, we write $X_E := X \dtimes{F} E$. The set of $E$-points is denoted by $X(E)$, which carries a topology when $E$ is a local field. The $F$-algebra (resp.\ field) of regular functions (resp.\ rational functions) on $X$ is denoted by $F[X]$ (resp.\ $F(X)$).

Unless otherwise specified, algebraic groups act on varieties on the right, and act on the left of function spaces by $\varphi \mapsto [g\varphi: x \mapsto \varphi(xg)]$. In particular, for any finite-dimensional $F$-vector space, we let $\GL(X)$ act on the right of $X$, although the scalar multiplication by $F$ is still on the left of $X$. The dual of a finite-dimensional vector space $X$ is denoted by $\check{X}$. If $\rho: G \to \GL(X)$ is a representation on $X$, its contragredient $\check{\rho}: G \to \GL(\check{X})$ is defined to render the canonical pairing $\lrangle{\cdot, \cdot}: \check{X} \times X \to F$ invariant.

Let $\Gm$ be the multiplicative $F$-group scheme. Let $G$ be a linear algebraic $F$-group where $F$ is any field. We set $\mathbf{X}^*(G) := \Hom_{\text{alg.grp} / F}(G, \Gm)$, which is an additive group. The derived subgroup of an algebraic group $G$ is denoted by $G_{\mathrm{der}}$. The center of $G$ is denoted by $Z_G$.

Suppose that $G$ is connected reductive and the variety $X$ is endowed with a $G$-action; we say $X$ is a \emph{$G$-variety}. We say a normal $G$-variety $X$ is \emph{spherical} if $X_{\overline{F}}$ has an open orbit under any Borel subgroup of $G_{\overline{F}}$; this is also known as \emph{absolutely sphericity} since we work over $\overline{F}$, an algebraic closure of $F$.

\paragraph*{Algebraic differential operators}
For a smooth variety $X$ over a field $F$ with characteristic zero, $\mathscr{D}_X$ will denote the Zariski sheaf of algebraic differential operators on $X$. The formation of $\mathscr{D}_X$ commutes with arbitrary field extensions $E|F$. Since we will mainly work with affine $X$, it is customary to consider the algebra $\mathcal{D}(X) = \Gamma(X, \mathscr{D}_X)$ of algebraic differential operators on $X$.

For more backgrounds about algebraic differential operators, we refer to \cite[\S 1.1]{BB93}.

If $X$ is a $G$-variety where $G$ is an algebraic group, then $G$ acts on $\mathcal{D}(X)$ by transport of structure, written as $D \xmapsto{g} {}^g D = gDg^{-1}$.

\paragraph*{Analysis}
The topological vector spaces are always over $\CC$ and locally convex. For a topological vector space $V$, we denote by $V^\vee := \Hom_{\text{cont}}(V, \CC)$ its continuous dual.

The space of jointly continuous bilinear forms on $V_1 \times V_2$ is denoted by $\mathrm{Bil}(V, W)$ where $V, W$ are topological vector spaces; see \cite[\S 41]{Tr67}.

Let $\Omega$ be a connected complex manifold and $V$ be a topological vector space. For a map of the form $Z: \Omega \to V^\vee$, written as $\lambda \mapsto Z_\lambda$, we say $Z$ is holomorphic if so is $\lambda \mapsto Z_\lambda(v)$ for each $v \in V$. Now suppose that $T$ is only defined off a nowhere-dense subset of $\Omega$. We say that $T$ is meromorphic if locally on $\Omega$ there exists a holomorphic function $L$, not identically zero, such that $\lambda \mapsto L(\lambda)T_\lambda$ is holomorphic. Two meromorphic families on $\Omega$ are identified if they agree off a nowhere-dense subset.

Let $R$ be an open subset of $\R^n$ for some $n$. A holomorphic function $f: R \times i\R \to \CC$ is said to be bounded on vertical strips if for each compact $C \subset R$, the restriction of $f$ to $C \times i\R$ is bounded.

The space of scalar-valued Schwartz--Bruhat functions on $X$ is denoted by $\Schw_0(X)$. Our conventions on Fourier transforms will be explained in \S\ref{sec:Fourier}; the version for half-densities will also be introduced.

\paragraph*{Representations}
When a group $H$ acts on some space $V$, we denote by $V^H$ the subspace of $G$-invariants in $V$.

Let $G$ be a connected reductive $\R$-group. Unless otherwise specified, the representations of $G(\R)$ are taken over $\CC$ and are tacitly assumed to be continuous. The representations under consideration in this article are mainly the SAF representations, also known as Casselman--Wallach representations; see \cite[p.46]{BK14}. We will also consider the smooth representations of $G(\R)$, for which we refer to \cite[\S 1]{Ca89} for the basic definitions.

Suppose that $\pi$ is such a representation of $G(\R)$. The central character of $\pi$, if it exists, will be denoted as $\omega_\pi: Z_G(\R) \to \CC^\times$. The underlying topological $\CC$-vector space of $\pi$ will be denoted as $V_\pi$.

For $G$ as above, we let $\mathfrak{g} := \Lie G$ and write $\mathcal{Z}(\mathfrak{g})$ for the center of the enveloping algebra $U(\mathfrak{g})$. Therefore $U(\mathfrak{g})$ acts on $V_\pi$ for any smooth $G(\R)$-representation $\pi$.

Assume $\pi$ is an SAF representation. For any maximal compact subgroup $K \subset G(\R)$, the space $V_\pi^{K\text{-fini}}$ of $K$-finite vectors in $V_\pi$ form a $(\mathfrak{g}, K)$-module. % Taking $K$-finite vectors gives an equivalence from the category of SAF representations to the category of Harish-Chandra $(\mathfrak{g}, K)$-modules.

\section{Prehomogeneous vector spaces}\label{sec:PVS}
\subsection{Relative invariants and regularity}\label{sec:relative-invariants}
We begin by reviewing the basic set-up about prehomogeneous vector spaces from \cite{Li18}; see also \cite[Chapter 6]{LiLNM} or \cite{Sa89, Ki03}. The following assumptions will remain in force throughout this article.

\begin{hypothesis}\label{hyp:PVS}
	Fix an additive character $\psi$ of $\R$. Let $G$ be a connected reductive $\R$-group, $X \neq \{0\}$ be a finite-dimensional $\R$-vector space and $\rho: G \to \GL(X)$ an algebraic homomorphism, through which $G$ acts on the right of $X$. Assume that
	\begin{itemize}
		\item there is a Zariski-open dense $G$-orbit in $X$, denoted hereafter as $X^+$;
		\item $\partial X := X \smallsetminus X^+$ is a hypersurface in $X$ (equivalently, $X^+$ is affine by \cite[Theorem 2.28]{Ki03});
		\item $X^+$ is a spherical homogeneous $G$-space, i.e.\ absolutely spherical by convention.
	\end{itemize}
\end{hypothesis}

Then $X^+(\R)$ is a union of finitely many $G(\R)$-orbits. The triplet $(G, \rho, X)$ forms a \emph{reductive prehomogeneous vector space} over $\R$. We say that a nonzero $f \in \R(X)$ is a \emph{relative invariant} if there exists $\omega \in \mathbf{X}^*(G)$ such that $f(xg) = \omega(g) f(x)$ for all $(x,g) \in X \times G$; the character $\omega$ is unique, called the eigencharacter of $f$.

Relative invariants on an arbitrary prehomogeneous vector space are automatically homogeneous, according to \cite[Corollary 2.7]{Ki03}.

If $f \in \R(X)$ is a relative invariant, then the logarithmic derivative $f^{-1} \dd f$ defines a $G$-equivariant morphism $X^+ \to \check{X}$. We say $f$ is \emph{non-degenerate} if $f^{-1} \dd f$ is dominant; see \cite[Definition 2.14]{Ki03}.

The general theory of prehomogeneous vector spaces affords the \emph{basic relative invariants} $f_1, \ldots, f_r \in \R[X]$, say with eigencharacters $\omega_1, \ldots, \omega_r \in \mathbf{X}^*(G)$ under $G$-action, which define irreducible codimension-one components of $\partial X$. Moreover,
\begin{align*}
	\mathbf{X}^*_\rho(G) & := \left\{ \omega \in \mathbf{X}^*(G) : \text{eigencharacter of some relative invariant}\right\} \\
	& = \bigoplus_{i=1}^r \Z \omega_i .
\end{align*}
It is known that $\{\omega_1, \ldots, \omega_r\}$ is uniquely determined, whereas the $f_i$ corresponding to $\omega_i$ is unique up to $\R^\times$. Call $\omega_1, \ldots, \omega_r$ the \emph{basic eigencharacters}. Every relative invariant is proportional to $f_1^{a_1} \cdots f_r^{a_r}$ for a unique $(a_1, \ldots, a_r) \in \Z^r$. The zero loci of basic relative invariants correspond to the irreducible components of $\partial X$. Therefore, $f_1^{a_1} \cdots f_r^{a_r} \in \R[X]$ if and only if $a_1, \ldots, a_r \geq 0$.

When $G$ is split, the facts above have been reviewed in \cite[\S 6.2]{LiLNM}. The general case follows by Galois descent from $\CC$ to $\R$; specifically, $\Gal(\CC|\R)$ permutes the irreducible components of $\partial X_{\CC}$ and
\[ \mathbf{X}^*_\rho(G) = \mathbf{X}^*_{\rho \otimes \CC}(G_{\CC})^{\Gal(\CC|\R)}. \]
Indeed, the only non-trivial part is that \textit{a priori}, to each $\omega \in \mathbf{X}^*_{\rho \otimes \CC}(G_{\CC})^{\Gal(\CC|\R)} \subset \mathbf{X}^*(G)$ corresponds only a relative invariant $f \in \CC(X)$ that is unique up to $\CC^\times$, but we may take $f \in \R(X)$ by the following technique: for every $\sigma \in \Gal(\CC|\R)$ let $c_\sigma \in \CC^\times$ be such that ${}^\sigma f = c_\sigma f$, so that $\sigma \mapsto c_\sigma$ is a $1$-cocycle; Hilbert's Theorem 90 then implies that there exists $c \in \CC^\times$ such that $cf$ is $\Gal(\CC|\R)$-invariant as desired. Hence $\omega \in \mathbf{X}^*_\rho(G)$.

According to \cite[Theorem 2.28]{Ki03} and our assumptions, $(G, \rho, X)$ is regular. Specifically,
\begin{itemize}
	\item $(\det \rho)^2 \in \mathbf{X}^*_\rho(G)$ and it corresponds to a non-degenerate relative invariant in $\R(X)$;
	\item the dual triplet $(G, \check{\rho}, \check{X})$ is regular prehomogeneous as well;
	\item $\mathbf{X}^*_{\rho}(G) = \mathbf{X}^*_{\check{\rho}}(G)$;
	\item every non-degenerate relative invariant $f \in \R(X)$ induces an isomorphism $f^{-1} \dd f: X^+ \rightiso \check{X}^+$ of homogeneous $G$-spaces (see \cite[Theorem 2.16]{Ki03}).
\end{itemize}
Again, for split $G$ these properties are reviewed in \cite[Theorem 6.2.4]{LiLNM}, and the general case follows by Galois descent. We summarize below.

\begin{proposition}\label{prop:dual-triplet}
	The dual triplet $(G, \check{\rho}, \check{X})$ also satisfies Hypothesis \ref{hyp:PVS}. We have $\mathbf{X}^*_{\rho}(G) = \mathbf{X}^*_{\check{\rho}}(G)$, and $X^+ \simeq \check{X}^+$ as homogeneous $G$-spaces.
\end{proposition}

\begin{proposition}\label{prop:rel-invariant-opposite}
	The basic eigencharacters for $(G, \check{\rho}, \check{X})$ are $\omega_1^{-1}, \ldots, \omega_r^{-1}$.
\end{proposition}
\begin{proof}
	Set $\Lambda^+_\rho := \sum_i \Z_{\geq 0} \omega_i$; a similar construction for $(G, \check{\rho}, \check{X})$ yields $\Lambda^+_{\check{\rho}}$. According to \cite[Proposition 2.21]{Ki03}, if $\omega \in \mathbf{X}^*_\rho(G)$ corresponds to a polynomial relative invariant for $X$, then so is $\omega^{-1}$ for $\check{X}$; the result is stated over $\CC$ in \text{loc.\ cit.}, but the case over $\R$ follows as explained earlier, by Hilbert's Theorem 90. Hence $- \Lambda^+_\rho \subset \Lambda^+_{\check{\rho}}$. By symmetry, $- \Lambda^+_\rho = \Lambda^+_{\check{\rho}}$.

	Let $V$ be the $\R$-vector space generated by the lattice $\mathbf{X}^*_\rho(G) = \mathbf{X}^*_{\check{\rho}}(G)$ of rank $r$. Both $\R_{\geq 0} \Lambda^+_\rho$ and $\R_{\geq 0} \Lambda^+_{\check{\rho}}$ are cones in $V$ generated by $r$ extremal rays (see eg.\ \cite[Proposition 1.20]{BG09}). The foregoing result implies that $\omega_1^{-1}, \ldots, \omega_r^{-1}$ generate the $r$ extremal rays of $\R_{\geq 0} \Lambda^+_{\check{\rho}}$. On the other hand, they are indivisible in $\mathbf{X}^*_\rho(G) = \mathbf{X}^*_{\check{\rho}}(G)$, hence they must be the minimal lattice points in these extremal rays, that is, the basic eigencharacters for $(G, \check{\rho}, \check{X})$.
\end{proof}

\begin{corollary}\label{prop:rel-invariant-pair}
	There exist non-degenerate polynomial relative invariants $f \in \R[X]$ and $\check{f} \in \R[\check{X}]$ such that
	\begin{enumerate}[(i)]
		\item $f, \check{f} \geq 0$ on $\R$-points;
		\item $\partial X = \{x: f(x) = 0 \}$ and $\partial \check{X} = \{ \check{x} : \check{f}(\check{x}) = 0 \}$;
		\item $f$, $\check{f}$ have opposite eigencharacters.
	\end{enumerate}
\end{corollary}
\begin{proof}
	Take $(a_1, \ldots, a_r) \in \Z_{\geq 1}^r$ and relative invariants $f$, $\check{f}$ with eigencharacters $\prod_{i=1}^r \omega_i^{a_i}$ and $\prod_{i=1}^r \omega_i^{-a_i}$. Proposition \ref{prop:rel-invariant-opposite} says that they are polynomials with zero loci $\partial X$ and $\partial \check{X}$, respectively. To ensure (i), one can replace $f$, $\check{f}$ by $f^2$, $\check{f}^2$.
\end{proof}

\subsection{Density bundles}\label{sec:densities}
Below is a review of the formalism of densities, following \cite[\S 3.1]{LiLNM}.

Let $Y$ be any real smooth manifold. Roughly speaking, the densities on $Y$ are objects which can be integrated. More generally, for each $t \in \R$ there is a real line bundle $\mathcal{L}^t$ of \emph{$t$-densities} on $Y$, with $\mathcal{L} := \mathcal{L}^1$. Denote by $C_c(Y, \mathcal{L}^t)$ the space of continuous sections of $\mathcal{L}^t$ over $Y$ of compact support. Likewise, we have the space $C^\infty(Y, \mathcal{L}^t)$ of $C^\infty$-sections of $\mathcal{L}^t$.

\begin{remark}
	Although $\mathcal{L}^t$ are real line bundles, we will mostly work with complex-valued sections and their integrations. We will also write $\mathcal{L}^t_Y$ to indicate the reference to $Y$.
\end{remark}

The line bundles $\mathcal{L}^t$ come with
\begin{compactitem}
	\item canonical pairings
	\[ \mathcal{L}^s \otimes \mathcal{L}^t \to \mathcal{L}^{s+t}, \quad s, t \in \R \]
	and a trivialization of $\mathcal{L}^0$, subject to the unity, associativity and commutativity constraints;
	\item the integration as a linear functional
	\[ \int_Y: C_c(Y, \mathcal{L}) \to \CC, \quad \xi \mapsto \int_Y \xi. \]
\end{compactitem}

These data can be constructed by the following recipe. Denote by $\Omega_Y$ the line bundle of differential $1$-forms on $Y$. To $\topwedge \Omega_Y$ corresponds the $\R^\times$-torsor $\mathcal{G}$ on $Y$, whose local sections are non-vanishing differential forms of top degree. Let $t \in \R$. Using the group homomorphism $|\cdot|^t: \R^\times \to \R^\times_{> 0} \subset \R$, we form the following real line bundle on $Y$:
\[ \mathcal{L}^t := \mathcal{G} \utimes{|\cdot|^t} \R. \]

Specifically, let $U \subset Y$ be an open subset and $\omega$ be a non-vanishing continuous section of $\topwedge \Omega_Y$ over $U$. It yields a section $|\omega|^t$ of $\mathcal{G} \utimes{|\cdot|^t} \R^\times_{> 0}$, whence a section of $\mathcal{L}^t$. In general, sections of $\topwedge \Omega_Y$ can be locally expressed as $f \omega$ where $f$ is a continuous function on $Y$; one defines unambiguously the section
\[ |f\omega|^t := |f| \cdot |\omega|^t  \]
of $\mathcal{L}^t$. We have $|\omega|^{s+t} = |\omega|^s |\omega|^t$, etc. The integration of $\xi \in C_c(Y, \mathcal{L})$ is then performed via local charts and partition of unity, reducing everything to Lebesgue integrals. In particular, every continuous section of $\mathcal{L}$ gives rise to a Radon measure on $Y$.

Consequently, it makes sense to define the $L^p$-space of sections of $\mathcal{L}^{1/p}$ as the completion of $C_c(Y, \mathcal{L}^{1/p})$ with respect to $\|\xi\|_{L^p} := \left( \int_Y |\omega|^p \right)^{1/p}$, where $1 \leq p \leq +\infty$.

Below are some further properties of $\mathcal{L}^t$.
\begin{itemize}
	\item Pull-back: this is compatible with the pull-back of differential forms. Given a morphism $\nu: Y \to Z$ and a section $\xi$ of $\mathcal{L}^t_Z$, we write $\nu^* \xi$ for the resulting section of $\mathcal{L}^t_Y$.
	
	For $t=0$ it is the pull-back of functions, and for any differential form $\omega$ of top degree on $Z$ we have $\nu^* |\omega|^t = |\nu^* \omega|^t$.
	\item For any open subset $U \subset Y$, we have $\mathcal{L}^t_Y |_U \simeq \mathcal{L}^t_U$; for any real analytic manifolds $Y_1, Y_2$, we have $\mathcal{L}^t_{Y_1 \times Y_2} \simeq \mathcal{L}^t_{Y_1} \boxtimes \mathcal{L}^t_{Y_2}$. Both isomorphisms are canonical.
	\item Integration of densities satisfies the formula of change of variables
	\[ \int_Y \nu^* \xi = \int_X \xi, \quad \xi \in C_c(Y, \mathcal{L}). \]
	\item If a Lie group $H$ acts on the right of $Y$, then the bundles $\mathcal{L}^t$ have canonical $H$-equivariant structures.
%	Namely we have a canonical isomorphism
%	\[ a^* \left( \mathcal{L}^t \right) \simeq \mathrm{pr}_1^* \left( \mathcal{L}^t \right), \]
%	subject to the usual cocycle condition, where $a$ (resp.\ $\mathrm{pr}_1$) is the action (resp.\ projection) map $Y \times H \to Y$.
\end{itemize}

Now let $(G, \rho, X)$ as in Hypothesis \ref{hyp:PVS}. Every $\Omega \in \topwedge \check{X}$ affords a translation-invariant $t$-density $|\Omega|^t$. For every $g \in \GL(X)$ we have $g^* |\Omega|^t = |g^* \Omega|^t = |\det g|^t |\Omega|^t$. Most often, we will encounter the case $t = \frac{1}{2}$, i.e.\ the half-densities. If necessary, one can get rid of half-densities by the following observation.

\begin{lemma}\label{prop:avoid-density}
	Let $\phi \in \R(X)$ be a relative invariant with eigencharacter $(\det\rho)^2$, and let $\Omega \in \topwedge \check{X} \smallsetminus \{0\}$. Then $|\phi|^{-1/4} |\Omega|^{1/2}$ is a $G(\R)$-invariant and nowhere vanishing half-density over $X^+(\R)$. Consequently, $\mathcal{L}^{1/2}$ can be equivariantly trivialized over $X^+(\R)$.
\end{lemma}
\begin{proof}
	This is just a restatement of \cite[Lemma 6.6.1]{Li18}.
\end{proof}

We caution the reader that for homogeneous $G$-spaces in general, the density bundles are not necessarily equivariantly trivializable.

\subsection{Fourier transform on Schwartz spaces}\label{sec:Fourier}
Given $(G, \rho, X)$ be as in Hypothesis \ref{hyp:PVS}, we follow the paradigm of \S\ref{sec:densities} to define the following spaces.
\begin{itemize}
	\item $C^\infty(X^+) := C^\infty \left(X^+, \mathcal{L}^{1/2} \right)$: the space of $C^\infty$ half-densities. It is a Fréchet space with respect to the standard topology as prescribed in \cite[\S 4.1]{LiLNM}; see also \cite[\S 10, Example I]{Tr67} for the scalar-valued case.
	
	The semi-norms in question involve a continuous metric on $\mathcal{L}^{1/2}$, whose choice is immaterial since we consider only its supremum over compact subsets. In our case, one can even trivialize $\mathcal{L}^{1/2}$ to get a more canonical choice.
	\item $L^2(X^+)$: the Hilbert space of $L^2$ half-densities on $X^+(\R)$. It is the same as $L^2(X)$.
	\item $\Schw(X)$: the Fréchet space of Schwartz--Bruhat sections in $L^2(X)$. The relation of $\Schw(X)$ to the usual scalar-valued Schwartz--Bruhat space $\Schw_0(X)$ is straightforward: $\Schw(X) = \Schw_0(X) |\Omega|^{1/2}$ for any $\Omega \in \topwedge \check{X} \smallsetminus \{0\}$. In particular, $\Schw(X)$ is a nuclear Fréchet space.
	
	The topology on $\Schw_0(X)$ is described in \cite[\S 10, Example IV]{Tr67}.
\end{itemize}

The spaces $C^\infty(X^+)$ and $\Schw(X)$ are smooth $G(\R)$-representations and $L^2(X^+)$ is a unitary $G(\R)$-representation.

Define the Fourier transform for half-densities $\mathcal{F}_\psi: \Schw(X) \rightiso \Schw(\check{X})$ as in \cite[\S 6.1]{LiLNM}: it is an isomorphism between Fréchet spaces, and extends to an isomorphism $L^2(X) \rightiso L^2(\check{X})$ satisfying $\|\mathcal{F}\xi\| = \sqrt{A(\psi)} \|\xi\|$ for some constant $A(\psi) > 0$; see the Definition \ref{def:A-psi}. Below is a recap of the formulas.

Let $\lrangle{\cdot, \cdot}: \check{X} \times X \to \R$ be the canonical pairing between $\check{X}, X$, and similarly for $\lrangle{\cdot, \cdot}: \topwedge \check{X} \times \topwedge X \to \R$. Given $\Omega \in \topwedge \check{X} \smallsetminus \{0\}$, we take the $\Psi \in \topwedge X$ with $\lrangle{\Omega, \Psi} = 1$ and define the Fourier transforms
\begin{equation}\label{eqn:Fourier-explicit}
	\begin{tikzcd}[row sep=tiny]
		\mathcal{F}_{\psi, |\Omega|}: \Schw_0(X) \arrow[r] & \Schw_0(\check{X}) \\
		\xi_0 \arrow[mapsto, r] \arrow[phantom, u, "\in" sloped] & { \left[ \check{x} \mapsto \displaystyle\int_{x \in X(\R)} \xi_0(x) \psi(\lrangle{\check{x}, x}) |\Omega|  \right] } \arrow[phantom, u, "\in" sloped] , \\
		\mathcal{F}_\psi: \Schw(X) \arrow[r] & \Schw(\check{X}) \\
		\xi = \xi_0 |\Omega|^{1/2} \arrow[mapsto, r] \arrow[phantom, u, "\in" sloped] & \mathcal{F}_{\psi, |\Omega|}(\xi_0) |\Psi|^{1/2} \arrow[phantom, u, "\in" sloped] .
	\end{tikzcd}
\end{equation}
It is readily seen that $\mathcal{F}_\psi$ is independent of the choice of $\Omega$. By working with half-densities, $\mathcal{F}$ becomes $G(\R)$-equivariant (see \cite[Theorem 6.1.5]{LiLNM}) and we do not have to choose Haar measures.

When there is no confusion about additive characters, we shall write $\mathcal{F}$ instead of $\mathcal{F}_\psi$.

Every $f \in \R[X]$ induces a continuous endomorphism $\xi \mapsto f\xi$ on $\Schw(X)$, namely by pointwise multiplication. On the other hand, every $\check{f} \in \R[\check{X}]$ can be viewed as a differential operator of constant coefficients on $X(\R)$, which can act on $\Schw(X)$ as follows: express $\xi \in \Schw(X)$ as $\xi = \xi_0 |\Omega|^{1/2}$ as before. Hence $\check{f} \xi_0$ make sense and we put
\begin{equation}\label{eqn:derivation-Schwartz}\begin{gathered}
	\check{f} \xi := \left( \check{f} \xi_0 \right) |\Omega|^{1/2} \; \in \Schw(X), \\
	\text{and same for } \xi \in C^\infty(X^+), \text{etc.}
\end{gathered}\end{equation}
This is clearly continuous in $\xi$ and independent of the choice of $\Omega$.

A key observation is that the $G$-action on $\R[\check{X}]$ coincides with the $G$-action on differential operators: indeed, it suffices to compare these actions on $\R[\check{X}]^{\deg = 1}$.

The same constructions also apply to the dual side. In particular, $\R[X]$ acts on $\Schw(\check{X})$ via differential operators of constant coefficients. In order to fix notations, we record the following common sense.

\begin{lemma}\label{prop:F-commute}
	There exists a constant $c(\psi) \in \sqrt{-1} \cdot \R^\times$, depending only on $\psi$, such that for every homogeneous $f \in \R[X]$ and every $\xi \in \Schw(X)$, we have
	\[ \mathcal{F}(f\xi) = c(\psi)^{\deg f} \cdot f \mathcal{F}(\xi). \]
\end{lemma}
\begin{proof}
	Choose volume forms and use \eqref{eqn:Fourier-explicit}, \eqref{eqn:derivation-Schwartz} to reduce to classical Fourier analysis.
\end{proof}

The next issue is the dependence on $\psi$. For $a \in \R^\times$, write $\psi_a(t) = \psi(at)$ and let $\nu_a: \check{X} \to \check{X}$ be the map $y \mapsto ay$. We have $\nu_a^* |\Psi|^{1/2} = |\nu_a^* \Psi|^{1/2} = |a|^{\dim X /2} |\Psi|^{1/2}$ for all $\Psi \in \topwedge X$. One infers from \eqref{eqn:Fourier-explicit} that
\begin{equation}\label{eqn:Fourier-a}\begin{aligned}
	(\mathcal{F}_{\psi_a, |\Omega|} \xi_0)(\check{x}) & = (\mathcal{F}_{\psi, |\Omega|} \xi_0)(a\check{x}), \quad \check{x} \in \check{X}, \\
	\mathcal{F}_{\psi_a} \xi & = |a|^{- \dim X /2} \nu_a^* \left( \mathcal{F}_\psi \xi \right).
\end{aligned}\end{equation}

To $\psi$ and $|\Omega|$ is associated the \emph{dual Haar measure} $|\Psi|'$ on $\check{X}$ characterized by
\[ \int_{\check{X}} \left| \mathcal{F}_{\psi, |\Omega|} \xi_0 \right|^2 |\Psi|' = \int_X \left| \xi_0 \right|^2 |\Omega|. \]
\begin{compactitem}
	\item If $|\Omega|$ is replaced by $t|\Omega|$ where $t \in \R_{> 0}$, then $|\Psi|'$ gets multiplied by $t^{-1}$;
	\item If $\psi$ is replaced by $\psi_a$, then $|\Psi|'$ gets multiplied by $|a|^{\dim X}$.
\end{compactitem}

\begin{definition}\label{def:A-psi}
	Let $\Omega \in \topwedge \check{X} \smallsetminus \{0\}$; take the $\Psi \in \topwedge X$ such that $\lrangle{\Psi, \Omega} = 1$. Define the $|\Psi|'$ as before, with respect to $|\Omega|$ and $\psi$. Set
	\[ A(\psi) := \frac{|\Psi|}{|\Psi|'} \; \in \R_{> 0}. \]
\end{definition}

By the foregoing discussions, $A(\psi)$ depends only on $\psi$ and $X$. Also,
\[ A(\psi_a) = |a|^{-\dim X} A(\psi), \quad a \in \R^\times . \]

\begin{example}\label{eg:psi-standard}
	The classical Plancherel's identity says that $\psi(t) = e^{2\pi i t}$ satisfies $A(\psi) = 1$.
\end{example}

Using the dual Haar measures $|\Omega|$ and $|\Psi|'$, the \emph{Fourier inversion formula} reads
\begin{equation*}
	\mathcal{F}_{-\psi, |\Psi|'} \mathcal{F}_{\psi, |\Omega|} = \identity_{\Schw_0(X)} .
\end{equation*}

\begin{proposition}\label{prop:F-inversion}
	For every choice of $\psi$, we have $\mathcal{F}_{-\psi} \mathcal{F}_\psi = A(\psi) \cdot \identity_{\Schw(X)}$.
\end{proposition}
\begin{proof}
	Take $\Omega$, $\Psi$, $|\Psi|'$ as above. Let $\xi = \xi_0 |\Omega|^{1/2} \in \Schw(X)$. Apply \eqref{eqn:Fourier-explicit} twice to see
	\begin{multline*}
		\mathcal{F}_{-\psi} \mathcal{F}_\psi \xi = \mathcal{F}_{-\psi} \left( \mathcal{F}_{\psi, |\Omega|} \xi_0 \cdot |\Psi|^{1/2} \right) = \left( \mathcal{F}_{-\psi, |\Psi|} \mathcal{F}_{\psi, |\Omega|} \xi_0 \right) \cdot |\Omega|^{1/2} \\
		=  \frac{|\Psi|}{|\Psi|'} \cdot \left( \mathcal{F}_{-\psi, |\Psi|'} \mathcal{F}_{\psi, |\Omega|} \xi_0 \right) \cdot |\Omega|^{1/2} = A(\psi) \xi ,
	\end{multline*}
	as asserted.
\end{proof}

\begin{remark}\label{rem:selfdual-F}
	These properties motivate us to define the self-dual version of $\mathcal{F}_\psi$, namely
	\[ \mathcal{F}^{\mathrm{sd}}_\psi := A(\psi)^{-1/2} \mathcal{F}_\psi . \]
	It satisfies $\mathcal{F}^{\mathrm{sd}}_{-\psi} \mathcal{F}^{\mathrm{sd}}_\psi = \identity_{\Schw(X)}$ and extends to a $G(\R)$-equivariant isometry $L^2(X) \rightiso L^2(\check{X})$.
\end{remark}

\section{Desiderata}\label{sec:desiderata}
Throughout this section, $(G, \rho, X)$ will be as in Hypothesis \ref{hyp:PVS}.

\subsection{Coefficients of representations}\label{sec:coeff-rep}
Let $\pi$ be an SAF representation in the sense of \cite{BK14}, also known as Casselman--Wallach representation; note that $V_\pi$ is nuclear. Following \cite[\S 4.1]{LiLNM}, we set
\[ \mathcal{N}_\pi(X^+) := \Hom_{G(\R)}(\pi, C^\infty(X^+)) \]
where the $\Hom_{G(\R)}$ is the continuous and $G(\R)$-equivariant $\Hom$-space between continuous representations.

For $\eta \in \mathcal{N}_\pi(X^+)$ and $v \in V_\pi$, we call $\eta(v) \in C^\infty(X^+)$ a \emph{generalize matrix coefficient} of $\pi$ on $X^+(\R)$, with values in half-densities.

\begin{remark}\label{rem:eta-0}
	Let $C^\infty(X^+; \CC)$ denote the usual topological vector space of $C^\infty$-functions on $X^+(\R)$. It is more common to consider scalar-valued generalized matrix coefficients arising from $\Hom_{G(\R)}(\pi, C^\infty(X^+; \CC))$, yet there is little difference: Lemma \ref{prop:avoid-density} furnishes the isomorphism
	\begin{equation*}\begin{aligned}
			\Hom_{G(\R)}(\pi, C^\infty(X^+)) & \rightiso  \Hom_{G(\R)}(\pi, C^\infty(X^+; \CC)) \\
			\eta & \mapsto \eta_0 := \eta \cdot |\phi|^{1/4} |\Omega|^{-1/2}
	\end{aligned}\end{equation*}
	with $\phi$, $\Omega$ as in Lemma \ref{prop:avoid-density}.
\end{remark}

Recall the following
\begin{theorem}
	The $\CC$-vector space $\mathcal{N}_\pi(X^+)$ is finite-dimensional.
\end{theorem}
\begin{proof}
	It suffices to show that $\Hom_{G(\R)}(\pi, C^\infty(O))$ is finite-dimensional for each $G(\R)$-orbit $O$, which is closed and open in $X^+(\R)$. This property is covered by \cite[Theorem A]{KO13} if $\pi$ is irreducible. For the general case, one can fix a maximal compact subgroup $K$, replace $C^\infty(O)$ by its scalar version $C^\infty(O; \CC)$ (see above) and pass to the Harish-Chandra module $V_\pi^{K\text{-fini}}$; the main result of \cite{KS16} then implies finiteness.
\end{proof}

Let $C^\infty_c(X^+) := C^\infty_c(X^+, \mathcal{L}^{1/2})$. As explained in \cite[\S 4.1]{LiLNM} or \cite[\S 13]{Tr67}, $C^\infty_c(X^+)$ carries a natural topology through
\[ C^\infty_c(X^+) = \varinjlim_{\substack{\Omega \subset X^+(\R) \\ \text{compact}}} C^\infty_\Omega(X^+) \]
where $C^\infty_\Omega(X^+) := \left\{ u \in C^\infty_c(X^+): \Supp(u) \subset \Omega \right\}$ carries the semi-norms given by suprema of derivatives, using any continuous metric on $\mathcal{L}^{1/2}$. It then becomes a smooth $G(\R)$-representation on an LF-space (= strict inductive limit of Fréchet spaces). The inclusion $C^\infty_c(X^+) \hookrightarrow \Schw(X)$ is equivariant and continuous. Upon choosing a volume form, these facts reduce to the well-known setting of scalar-valued functions. Elements of $C_c^\infty(X^+)^\vee$ are nothing but \emph{distributions} on the open subset $X^+(\R)$ of $X(\R)$, which fits into the classical picture if we choose volume forms.

Notice that $G(\R)$ acts linearly on $C_c^\infty(X^+)^\vee$, and the inclusion map $C^\infty(X^+) \hookrightarrow C^\infty_c(X^+)^\vee$ is $G(\R)$-equivariant.

Since $\dim_{\CC} \mathcal{N}_\pi(X^+)$ is finite, one can talk about holomorphic or meromorphic families inside $\mathcal{N}_\pi(X^+)$ unambiguously. Fix a maximal compact subgroup $K \subset G(\R)$.

\begin{definition}
	For any commutative ring $A$, set $\Lambda_A := \mathbf{X}^*_\rho(G) \otimes_{\Z} A$.
\end{definition}

\begin{lemma}\label{prop:holomorphy-test}
	Let $\{\eta_\omega\}_{\omega \in \Omega}$ be a family of elements in $\mathcal{N}_\pi(X^+)$, where $\Omega$ is a connected complex manifold. Let $\lambda \in \Lambda_{\CC}$. The following are equivalent:
	\begin{enumerate}[(i)]
		\item $\{\eta_\omega\}_{\omega \in \Omega}$ is a holomorphic family in $\mathcal{N}_\pi(X^+)$;
		\item $\int_{X^+(\R)} \eta_\omega(v)|f|^\lambda \xi$ is holomorphic in $\omega$ for all $v \in V_\pi^{K\text{-fini}}$ and $\xi \in C^\infty_c(X^+)$.
	\end{enumerate}
\end{lemma}
\begin{proof}
	Clearly, (i) implies (ii). Now assume (ii) and write $Z_\lambda(\eta, v, \xi) := \int_{X^+(\R)} \eta(v)|f|^\lambda \xi$ (temporarily, but see Proposition \ref{prop:zeta-restricted}). Observe that if $\eta \in \mathcal{N}_\pi(X^+)$ satisfies $Z_\lambda(\eta, v, \xi) = 0$ for all $(v, \xi) \in V_\pi^{K\text{-fini}} \times C^\infty_c(X^+)$, then $\eta(v) = 0$ for all $v \in V_\pi^{K\text{-fini}}$, hence $\eta = 0$ by its continuity.
	
	Since $\dim_{\CC} \mathcal{N}_\pi(X^+)$ is finite, these linear functionals generate $\mathcal{N}_\pi(X^+)^\vee$ and there is a finite subset $F \subset V_\pi^{K\text{-fini}} \times C^\infty_c(X^+)$ such that
	\begin{align*}
		\mathcal{N}_\pi(X^+) & \hookrightarrow \CC^F \\
		\eta & \mapsto \left( Z_\lambda(\eta, v, \xi) \right)_{(v, \xi) \in F}.
	\end{align*}
	As $\omega \mapsto Z_\lambda(\eta_\omega, v, \xi)$ is holomorphic for each $(v, \xi)$, the property (i) follows at once.
\end{proof}

\begin{corollary}\label{prop:holomorphy-twist}
	Let $\{\eta_\omega\}_{\omega \in \Omega}$ be a holomorphic family of elements inside $\mathcal{N}_\pi(X^+)$, where $\Omega$ is a connected complex manifold. For every $\lambda \in \Lambda_{\CC}$, the family $\{ \eta_\omega |f|^\lambda \}_{\omega \in \Omega}$ inside $\mathcal{N}_{\pi \otimes |\omega|^\lambda}(X^+)$ is also holomorphic.
\end{corollary}
\begin{proof}
	Apply the characterization (ii) in Lemma \ref{prop:holomorphy-test}.
\end{proof}

Using the embedding $X^+ \hookrightarrow X$, we let $\mathcal{D}(X^+)$ act on the left of $C^\infty(X^+)$ as follows.

\begin{definition}
	Let $u = u_0 |\Omega|^{1/2} \in C^\infty(X^+)$ where $u_0 \in C^\infty(X^+; \CC)$ and $\Omega \in \topwedge \check{X} \smallsetminus \{0\}$. For $D \in \mathcal{D}(X^+)$, set
	\[ Du := (Du_0) \cdot |\Omega|^{1/2}. \]
	This makes $C^\infty(X^+)$ into a left $\mathcal{D}(X^+)$-module, independently of the choice of $\Omega$. The recipe is compatible with \eqref{eqn:derivation-Schwartz}.
\end{definition}

Write $D \mapsto {}^g D = gDg^{-1}$ for the left action of $g \in G$ on differential operators, and similarly for functions, volume forms, etc. It is routine to see that ${}^g u = {}^g u_0 \cdot |{}^g \Omega|^{1/2}$ and
\begin{align*}
	{}^g (Du) & = {}^g (Du_0) \cdot |{}^g \Omega|^{1/2} = ({}^g D) ({}^g u_0) \cdot |{}^g \Omega|^{1/2} \\
	& = ({}^g D) ({}^g u).
\end{align*}
The case of real-analytic differential operators is completely analogous.

\begin{definition}\label{def:D-action-N}
	Make $\mathcal{N}_\pi(X^+)$ into a left $\mathcal{D}(X^+)^G$-module by setting $D\eta$ to be $v \mapsto D(\eta(v))$, for all $\eta \in \mathcal{N}_\pi(X^+)$ and $D \in \mathcal{D}(X^+)^G$. More generally, $\mathcal{N}_\pi(X^+)$ is a left module under the ring of $G(\R)$-invariant real-analytic differential operators on $X^+(\R)$.
\end{definition}

All the foregoing constructions apply to the dual triplet $(G, \check{\rho}, \check{X})$ as well. By Proposition \ref{prop:dual-triplet} and the canonicity of density bundles, for any given $\pi$ we have can take a non-degenerate relative invariant $f$ to obtain an isomorphism
\[\begin{tikzcd}[row sep=tiny]
	\mathcal{N}_\pi(X^+) \arrow[r, "\sim"] & \mathcal{N}_\pi(\check{X}^+) \\
	\eta \arrow[mapsto, r] & (f^{-1} \dd f)_* \circ \eta
\end{tikzcd}\]
where $(f^{-1} \dd f)_*: C^\infty(X^+) \rightiso C^\infty(\check{X}^+)$ is the transport of structure applied to half-densities.

\subsection{Statement of the main theorems}\label{sec:main-thm}
The following constructions and statements are extracted from \cite{LiLNM, Li18}.

\begin{definition}
	Choose basic relative invariants $f_1, \ldots, f_r \in \R[X]$ as in \S\ref{sec:relative-invariants}, with eigencharacters $\omega_1, \ldots, \omega_r$. For every $\lambda = \sum_{i=1}^r \omega_i \otimes \lambda_i \in \Lambda_{\CC}$, we write
	\[ |f|^\lambda := \prod_{i=1}^r |f_i|^{\lambda_i}, \quad |\omega|^\lambda := \prod_{i=1}^r |\omega_i|^{\lambda}, \]
	so that $|f|^\lambda: X(\R) \to \R_{\geq 0}$ has $G(\R)$-eigencharacter $|\omega|^\lambda$.
\end{definition}

For $\lambda = \sum_{i=1}^r \omega_i \otimes \lambda_i \in \Lambda_{\CC}$ and $\kappa = \sum_{i=1}^r \omega_i \otimes \kappa_i \in \Lambda_{\R}$, the notation $\Re(\lambda) \relgeq{X} \kappa$ signifies that $\Re(\lambda_i) \geq \kappa_i$ for all $i$; the notation $\Re(\lambda) \relgg{X} 0$ signifies that $\Re(\lambda_i) \gg 0$ for all $i$.

\begin{definition}[Generalized zeta integral]\label{def:zeta}
	Let $\pi$ be an SAF representation of $G(\R)$. For all $\eta \in \mathcal{N}_\pi(X^+)$, $v \in V_\pi$, $\xi \in \Schw(X)$ and $\lambda \in \Lambda_{\CC}$ with $\Re(\lambda) \relgg{X} 0$ (see the discussion below), set
	\[ Z_\lambda(\eta, v, \xi) := \int_{X^+(\R)} \eta(v) |f|^\lambda \xi. \]
	The integrand is a density on $X^+(\R)$, hence the integral makes sense. If we write $\xi = \xi_0 |\Omega|^{1/2}$ and $\eta(v) = \eta_0(v) |\phi|^{-1/4} |\Omega|^{1/2}$ (as in Remark \ref{rem:eta-0}), arrange that $|\phi|^{1/4} = |f|^{\lambda_0}$ for some $\lambda_0 \in \frac{1}{4} \Lambda_{\Z}$, and consider the invariant measure $\dd \mu := |\phi|^{-1/2} |\Omega|^{1/2}$ on $X^+(\R)$, then
	\begin{align*}
		Z_\lambda(\eta, v, \xi) & = \int_{X(\R)} \eta_0(v) |f|^{\lambda - \lambda_0} \xi_0 |\Omega| \\
		& = \int_{X^+(\R)} \eta_0(v) |f|^{\lambda + \lambda_0} \xi_0 \dd \mu.
	\end{align*}
\end{definition}

We will also view $Z_\lambda(\eta, \cdot, \cdot)$ as a family of bilinear forms in $(v, \xi)$. Implicit in the definition above is the convergence of $Z_\lambda(\eta, v, \xi)$ for $\Re(\lambda) \relgg{X} 0$. This is made precise in the following main result.

\begin{theorem}\label{prop:convergence}
	There is a constant $\kappa = \kappa(\pi)\in \Lambda_{\R}$, depending only on $\pi$ and $(G, \rho, X)$, such that the integral in Definition \ref{def:zeta} converges whenever $\Re(\lambda) \relgeq{X} \kappa$, for all $v, \xi$ and $\eta \in \mathcal{N}_\pi(X^+)$. Inside this range of convergence,
	\begin{enumerate}[(i)]
		\item $Z_\lambda(\eta, v, \xi)$ is jointly continuous in $(v, \xi)$;
		\item $\lambda \mapsto Z_\lambda(\eta, v, \xi)$ is holomorphic, when viewed as a function valued in $\mathrm{Bil}(V_\pi, \Schw(X)) \simeq \left(V_\pi \hat{\otimes} \Schw(X) \right)^\vee$;
		\item $Z_\lambda(\eta, v, \xi)$ is bounded in vertical strips as a function in $\lambda$ for any pair $(v, \xi)$.
	\end{enumerate}
	Here $\mathrm{Bil}(V_\pi, \Schw(X))$ stands for the space of jointly continuous bilinear forms, $\hat{\otimes}$ stands for the completed tensor product for nuclear spaces, and $(\cdots)^\vee$ stands for the continuous dual.
\end{theorem}

For the meaning of holomorphy for $\left(V_\pi \hat{\otimes} \Schw(X) \right)^\vee$-valued functions, see \S\ref{sec:conventions}.

\begin{remark}
	The theory is vacuous unless $\mathcal{N}_\pi(X^+) \neq \{0\}$, i.e.\ unless the representation $\pi$ is \emph{distinguished} by $X^+$.
\end{remark}

\begin{theorem}\label{prop:meromorphy}
	The zeta integrals $Z_\lambda$ extends meromorphically to all $\lambda \in \Lambda_{\CC}$. More precisely, fix $\eta \in \mathcal{N}_\pi(X^+)$ and $a \in \Lambda_{\Z}$ with $a \relg{X} 0$, then there exist
	\begin{itemize}
		\item $t \in \Z_{\geq 1}$ and affine hyperplanes $H_1, \ldots, H_t$ in $\Lambda_{\CC}$ whose vectorial parts $\vec{H_i}$ are all $\Q$-rational,
		\item a holomorphic function $\lambda \mapsto L(\eta, \lambda)$ on $\Lambda_{\CC}$, not identically zero,
	\end{itemize}
	such that
	\begin{itemize}
		\item the function in $\lambda$
		\[ LZ_\lambda(\eta, v, \xi) := L(\eta, \lambda) Z_\lambda(\eta, v, \xi), \quad (v, \xi) \in V_\pi \times \Schw(X), \]
		initially defined only for $\Re(\lambda) \relgeq{X} \kappa$, extends to a holomorphic function
		\[ \left[ \lambda \mapsto LZ_\lambda(\eta, \cdot, \cdot) \right]: \Lambda_{\CC} \to \mathrm{Bil}(V_\pi, \Schw(X)) \simeq \left( V_\pi \hat{\otimes} \Schw(X) \right)^\vee \]
		which yields the meromorphic continuation of $Z_\lambda$;
		\item the polar set of $Z_\lambda$ is a union of translates $H_i - ma$, for various $1 \leq i \leq t$ and $m \in \Z_{\geq 1}$.
	\end{itemize}
	Furthermore,
	\begin{enumerate}[(i)]
		\item one can take $L(\eta, \lambda) = \prod_{i=1}^m \Gamma(\alpha_i(\lambda))^{-1}$ where $\alpha_1, \ldots, \alpha_m$ are certain affine functions on $\Lambda_{\CC}$, whose gradients are among $\vec{H}_1, \ldots, \vec{H}_t$;
		\item the jointly continuous bilinear form $LZ_\lambda(\eta, \cdot, \cdot)$ on $(\pi \otimes |\omega|^\lambda) \times \Schw(X)$ is $G(\R)$-invariant for all $\lambda \in \Lambda_{\CC}$.
	\end{enumerate}
\end{theorem}

In view of the results above, $Z_\lambda(\eta, v, \xi)$ may now be viewed as a meromorphic family of trilinear forms in $(\eta, v, \xi)$ on the whole $\Lambda_{\CC}$, which are jointly continuous (recall that $\dim_{\CC} \mathcal{N}_\pi(X^+) < +\infty$).

For the dual triplet $(G, \check{\rho}, \check{X})$, we also form the space $\mathcal{N}_\pi(\check{X}^+)$. Since $\mathbf{X}^*_\rho(G) = \mathbf{X}^*_{\check{\rho}}(G)$, the same $\Lambda_{\CC}$ parameterizes both zeta integrals $Z_\lambda$ (on $X$) and $\check{Z}_\lambda$ (on $\check{X}$). Fix basic relative invariants $\check{f}_1, \ldots, \check{f}_r$ for $(G, \check{\rho}, \check{X})$ and write $|\check{f}|^\lambda := \prod_{i=1}^r |\check{f}_i|^{\lambda_i}$, by which we define the zeta integrals $\check{Z}_\lambda$. Observe that $|\omega|^\lambda: G(\R) \to \R^\times_{>0}$ does not depend on the eigencharacters $\omega_1, \ldots, \omega_r$: it is determined by $\lambda \in \Lambda_{\CC} \subset \mathbf{X}^*(G) \otimes \CC$, thus the notation works uniformly for both $(G, \rho, X)$ and $(G, \check{\rho}, \check{X})$.

Now comes the local functional equation. For $\pi$ as above and $\eta \in \mathcal{N}_\pi(X^+)$, put
\begin{equation}\label{eqn:eta-lambda}
	\eta_\lambda: v \mapsto \eta(v) |f|^\lambda
\end{equation}
which belongs to $\mathcal{N}_{\pi \otimes |\omega|^\lambda}(X^+)$ for all $\lambda \in \Lambda_{\CC}$. The same definition applies to $\check{X}^+$ as well.

\begin{theorem}\label{prop:LFE}
	Assume that $\pi$ is an SAF representation with central character. There is a meromorphic family of linear maps $\gamma(\pi, \lambda): \mathcal{N}_\pi(\check{X}^+) \to \mathcal{N}_\pi(X^+)$ (i.e.\ its matrix entries are meromorphic in $\lambda$), depending on $\psi$, such that
	\[ \check{Z}_\lambda\left( \check{\eta}, v, \mathcal{F}\xi \right) = Z_\lambda\left( \gamma(\lambda, \pi)(\check{\eta}), v, \xi \right). \]
	for all $\check{\eta} \in \mathcal{N}_\pi(\check{X}^+)$, $v \in V_\pi$, $\xi \in \Schw(X)$ and all $\lambda \in \Lambda_{\CC}$ off the poles of $Z_\lambda$ and $\check{Z}_\lambda$.
	
	Moreover:
	\begin{enumerate}[(i)]
		\item $\gamma(\pi, \lambda)$ is uniquely characterized by this equality;
		\item if $L(\check{\eta}, \lambda)$ is as in Theorem \ref{prop:meromorphy}, then $L(\check{\eta}, \lambda) \gamma(\pi, \lambda)$ is holomorphic in $\lambda$;
		\item $\gamma(\pi, \lambda + \mu)(\check{\eta})_\lambda = \gamma(\pi \otimes |\omega|^\lambda, \mu)(\check{\eta}_\lambda)$ for all $\mu, \lambda, \check{\eta}$; see \eqref{eqn:eta-lambda}.
	\end{enumerate}
\end{theorem}

The Theorems \ref{prop:convergence}, \ref{prop:meromorphy} and \ref{prop:LFE} are stated as axioms in \cite[Chapter 4]{LiLNM} in an abstract setting; as to the special case of prehomogeneous vector spaces, see also \cite[Chapter 6]{LiLNM}. Theorems \ref{prop:convergence} and \ref{prop:meromorphy} have also appeared in \cite{Li18} when $X^+$ is an ``essentially symmetric'' homogeneous $G$-space. It is routine to see that these assertions are unaffected by the choice of basic relative invariants $f_1, \ldots, f_r$; see \cite[Lemma 4.5.4]{LiLNM} for a precise statement.

The proofs of the Theorems above will occupy the rest of this article.

We close this subsection by addressing the dependence of $\gamma$-factors on the additive character $\psi$ of $\R$. Let $a \in \R^\times$ and define:
\begin{itemize}
	\item $d(\lambda) = \sum_{i=1}^r \lambda_i \cdot \deg \check{f}_i$ for all $\lambda = \sum_{i=1}^r \omega_i \otimes \lambda_i \in \Lambda_{\CC}$;
	\item $m_a: C^\infty(\check{X}^+) \to C^\infty(\check{X}^+)$ is the pull-back along the automorphism $\nu_{a^{-1}}: y \mapsto a^{-1} y$ of $\check{X}^+(\R)$, so that $m_a$ is $G(\R)$-equivariant.
\end{itemize}
We write $\gamma(\pi, \lambda) = \gamma(\pi, \lambda; \psi)$, $\mathcal{F} = \mathcal{F}_\psi$, etc. Denote the $\gamma$-factor defined relative to $(G, \check{\rho}, \check{X})$ and $\psi$ by $\check{\gamma}(\pi, \lambda; \psi)$.

\begin{proposition}\label{prop:gamma-symmetry}
	For any $a \in \R^\times$, let $\psi_a$ be the additive character $x \mapsto \psi(ax)$ of $\R$. Then
	\begin{enumerate}[(i)]
		\item $\check{\gamma}(\pi, \lambda; \psi_{-1})\gamma(\pi, \lambda; \psi) = A(\psi) \cdot \identity_{\mathcal{N}_\pi(\check{X}^+)}$ as meromorphic families in $\lambda$, where $A(\psi) \in \R_{> 0}$ is as in Definition \ref{def:A-psi};
		\item $\gamma(\pi, \lambda; \psi_a) = |a|^{-d(\lambda) - \frac{1}{2} \dim X} \gamma(\pi, \lambda; \psi) \circ m_a$ for all $a \in \R^\times$, as meromorphic families in $\lambda$.
	\end{enumerate}
\end{proposition}
\begin{proof}
	Both assertions rely on the uniqueness of $\gamma$-factors in Theorem \ref{prop:LFE}. Assertion (i) results from Proposition \ref{prop:F-inversion}. As to (ii), we apply \eqref{eqn:Fourier-a} to see that when $\Re(\lambda) \relgg{\check{X}} 0$,
	\begin{multline*}
		\check{Z}_\lambda\left( \check{\eta}, v, \mathcal{F}_{\psi_a} \xi \right) = |a|^{-\dim X /2} \int_{\check{X}^+(\R)} \check{\eta}(v) |\check{f}|^\lambda \nu_a^* (\mathcal{F}_\psi \xi) \\
		= |a|^{-\dim X /2} \int_{\check{X}^+(\R)} \nu_{a^{-1}}^* \left( \check{\eta}(v) | \check{f} |^\lambda \right) \mathcal{F}_\psi \xi \quad \text{($\because$ change of variables)} \\
		= |a|^{-d(\lambda) - (\dim X /2)} \int_{\check{X}^+(\R)} m_a (\check{\eta}(v)) \cdot |\check{f}|^\lambda \cdot \mathcal{F}_\psi \xi \\
		= Z_\lambda \left( |a|^{-d(\lambda) - (\dim X /2)} \gamma(\lambda, \pi; \psi) \circ m_a (\check{\eta}), v, \xi \right).
	\end{multline*}
	The equality extends by meromorphic continuation, and (ii) follows.
\end{proof}

\begin{remark}\label{rem:sd-version}
	If we define $\gamma^{\mathrm{sd}}(\pi, \lambda; \psi) := A(\psi)^{-1/2} \gamma(\pi, \lambda; \psi)$, i.e.\ by replacing $\mathcal{F}_\psi$ by its ``self-dual'' version $\mathcal{F}^{\mathrm{sd}}_\psi := A(\psi)^{-1/2} \mathcal{F}_\psi$ (Remark \ref{rem:selfdual-F}) in the characterization of $\gamma$-factors, the conclusions become
	\begin{enumerate}[(i)]
		\item $\check{\gamma}^{\mathrm{sd}}(\pi, \lambda; \psi_{-1})\gamma^{\mathrm{sd}}(\pi, \lambda; \psi) = \identity_{\mathcal{N}_\pi(\check{X}^+)}$,
		\item $\gamma^{\mathrm{sd}}(\pi, \lambda; \psi_a) = |a|^{-d(\lambda)} \gamma^{\mathrm{sd}}(\pi, \lambda; \psi) \circ m_a$,
	\end{enumerate}
	where in (ii) we used $A(\psi_a) = |a|^{-\dim X} A(\psi)$. In particular $\gamma^{\mathrm{sd}}(\pi, 0; \psi_a) = \gamma^{\mathrm{sd}}(\pi, 0; \psi) \circ m_a$. This mirrors the behavior of local root numbers over $\R$ described in \cite[(3.6.6)]{Ta79}, which is also formulated in terms of self-dual Haar measures. See also Remark \ref{rem:GJ-gamma-behavior}.
	
	An equivalence way is to keep the formula for $\mathcal{F}_\psi$, but renormalize the Haar measure on $\R$ to be self-dual with respect to $\psi$; this also normalizes the integration of densities. See \cite[Lemma 6.1.4]{LiLNM}.
\end{remark}

\subsection{Examples}\label{sec:examples}

In all the examples below, we fix $\Omega \in \topwedge \check{X}$, $\Psi \in \topwedge X$ with $\lrangle{\Omega, \Psi} = 1$.

\begin{example}[Sato--Shintani]\label{eg:SS}
	Let $(G, \rho, X)$ be as in Hypothesis \ref{hyp:PVS}, but take $\pi = \mathbf{1}$, the trivial representation. Decompose $X^+(\R)$ into $G(\R)$-orbits $\bigsqcup_{i=1}^k O_i$. For $1 \leq i \leq k$, let $c_i$ be the function on $X^+(\R)$ which is $1$ on $O_i$ and zero elsewhere. Take an invariant half-density $|\phi|^{-1/4} |\Omega|^{1/2}$ as in Lemma \ref{prop:avoid-density}, with $|\phi|^{1/4} = |f|^{\lambda_0}$ where $\lambda_0 \in \frac{1}{4} \Lambda_{\Z}$ (see Definition \ref{def:zeta}). Then
	\[\begin{tikzcd}[row sep=tiny]
		\CC^k \arrow[r, "\sim"] & \mathcal{N}_{\mathbf{1}}(X^+) \\
		(0, \ldots, \underbracket{\; 1 \;}_{i\text{-th}}, \ldots, 0) \arrow[mapsto, r] & \eta_i := c_i |\phi|^{-1/4} |\Omega|^{1/2}.
	\end{tikzcd}\]
	By choosing a non-degenerate relative invariant to obtain $X^+ \rightiso \check{X}^+$, the $G(\R)$-orbits in $X^+(\R)$ and $\check{X}^+(\R)$ are in bijection, both labeled by $\{1, \ldots, k\}$. In particular we can define $\check{\eta}_1, \ldots, \check{\eta}_k$.

	For $\phi$, $\lambda_0$ as above, $1 \leq i \leq k$ and $\Re(\lambda) \relgg{X} 0$, we obtain
	\begin{equation*}
		Z_\lambda\left( \eta_i, 1, \xi \right) = \int_{O_i} |f|^{\lambda - \lambda_0} \xi_0 |\Omega|.
	\end{equation*}
	By recalling the definition of $\lambda_0$ and identifying $\Lambda_{\CC}$ with $\CC^r$ by the basis $\omega_1, \ldots, \omega_r$, we see that $Z_\lambda\left( \eta_i, 1, \xi \right)$ equals the Archimedean local zeta integral $Z_i(\lambda + \lambda_0, \xi_0)$ defined in \cite[\S 1.4]{Sa89}.

	For $\check{X}$ we have $\check{Z}_i(\cdots)$ as well. Observe that
	\begin{compactitem}
		\item the avatar of $\lambda_0$ for $\check{X}$ is $-\lambda_0$;
		\item $\mathbf{X}^*_\rho(G) \otimes \CC = \mathbf{X}^*_{\check{\rho}}(G) \otimes \CC$, but their isomorphisms to $\CC^r$ induced by basic eigencharacters differ by $-1$, by Proposition \ref{prop:rel-invariant-opposite}.
	\end{compactitem}

	Use the standard additive character $\psi$ in Example \ref{eg:psi-standard} to perform Fourier transform. Together with the observations above, the Theorem $\mathbf{R}$ in \cite[p.471]{Sa89} gives meromorphic functions $\left( \Gamma_{ij} \right)_{1 \leq i,j \leq k}$ such that for all $\xi = \xi_0 |\Psi|^{1/2} \in \Schw(\check{X})$,
	\[ Z_i \left( \lambda, \mathcal{F} \xi_0 \right) = \sum_{j=1}^k \Gamma_{ij}(\lambda - 2\lambda_0) \check{Z}_j \left( \lambda - 2\lambda_0, \xi_0 \right). \]
	This can be rewritten as
	\[ Z_{\lambda - \lambda_0} \left( \eta_i, 1, \mathcal{F} \xi \right) = \sum_{j=1}^k \Gamma_{ij}(\lambda - 2\lambda_0) \check{Z}_{\lambda - \lambda_0} \left( \check{\eta}_j, 1, \xi \right). \]
	
	In our framework, $\check{\gamma}(\lambda, \mathbf{1}; \psi)$ is thus represented by the matrix $\left( \Gamma_{ij}(\lambda - \lambda_0) \right)_{1 \leq i, j \leq k}$ of meromorphic functions, with respect to the bases $\{\eta_i\}_i$ and $\{\check{\eta}_j\}_j$. Proposition \ref{prop:gamma-symmetry} implies that $\gamma(\lambda, \mathbf{1}; \psi) = \gamma(\lambda, \mathbf{1}, \psi_{-1}) \circ m_{-1}$ is represented by the matrix $\left( \Gamma_{ij}(\lambda - \lambda_0) \right)_{1 \leq i, j \leq k}^{-1} \cdot P_{-1}$, where $P_{-1}$ is the matrix corresponding to the permutation of the $G(\R)$-orbits in $\check{X}^+(\R)$ induced by $y \mapsto -y$.
\end{example}

\begin{example}[Godement--Jacquet]\label{eg:GJ}
	Let $D$ be a central simple $\R$-algebra of dimension $n^2$ and let $G = D^\times \times D^\times$ act on $X := D$, by
	\begin{equation*}\begin{tikzcd}
		x \arrow[mapsto, r, "{(g, h)}"] & h^{-1} x g.
	\end{tikzcd}\end{equation*}
	This is a regular prehomogeneous vector space with $X^+ = D^\times$, which is spherical (in fact, the ``group case'' of a symmetric space). The relative invariants are generated by the reduced norm $\mathrm{Nrd}$, up to $\R^\times$. Accordingly, $\mathbf{X}^*_\rho(G)$ is generated by $(g, h) \mapsto \mathrm{Nrd}(h)^{-1} \mathrm{Nrd}(g)$.
	
	We may identify $X$ with $\check{X}$ via the perfect pairing $(x, y) \mapsto \mathrm{Trd}(xy)$ on $X \times X$, where $\mathrm{Trd}$ is the reduced trace. As discussed in \cite[Lemma 6.4.1]{LiLNM}, $(G, \check{\rho}, \check{X})$ then becomes $X$ with the flipped action $x \xmapsto{(g, h)} g^{-1} x h$, and $(G, \rho, X)$ is regular. In fact $\mathrm{Nrd}$ is non-degenerate, and the induced equivariant isomorphism $X^+ \rightiso \check{X}^+$ is $x \mapsto x^{-1}$; cf.\ \cite[Proposition 6.4.2]{LiLNM}. We still write $X^+$ and $\check{X}^+$ in order to distinguish the $G$-actions. Note that $\mathrm{Nrd}$ is a basic relative invariant for both $X$ and $\check{X}$, but with opposite eigencharacters.

	The irreducible SAF representations $\pi$ with $\mathcal{N}_\pi(X^+) \neq \{0\}$ take the form $\sigma \boxtimes \check{\sigma}$, where $\sigma$ is an irreducible SAF representation of $D^\times(\R)$. Ditto for $\mathcal{N}_\pi(\check{X}^+)$. Note that $|\det|^{-n} |\Omega|$ defines a Haar measure on $D^\times(\R)$. The spaces $\mathcal{N}_{\sigma \boxtimes \check{\sigma}}(X^+)$ and $\mathcal{N}_{\sigma \boxtimes \check{\sigma}}(\check{X}^+)$ are spanned respectively by matrix coefficient maps
	\begin{align*}
		\eta_{\Omega}: v \otimes \check{v} & \mapsto \lrangle{\check{v}, \pi(\cdot) v} |\det|^{-n/2} |\Omega|^{1/2}, \\
		\check{\eta}_{\Psi}: v \otimes \check{v} & \mapsto \lrangle{\check{\pi}(\cdot) \check{v}, v} |\det|^{-n/2} |\Psi|^{1/2}.
	\end{align*}

	With these choices, we write $\xi = \xi_0 |\Omega|^{1/2} \in \Schw(X)$, $\check{\xi} = \check{\xi}_0 |\Psi|^{1/2} \in \Schw(\check{X})$, and let $Z^{\mathrm{GJ}}(\cdots)$ (resp.\ $\gamma^{\mathrm{GJ}}(\cdots)$) stand for the usual Godement--Jacquet zeta integrals in \cite[(15.4.3)]{GH11-2} (resp.\ the usual Godement--Jacquet $\gamma$-factors). Here we choose the standard additive character $\psi$ as in Example \ref{eg:psi-standard}. It turns out that
	\begin{align*}
		Z_\lambda \left( \eta_{\Omega}, v \otimes \check{v}, \xi \right) & = Z^{\mathrm{GJ}}\left( \lambda + \frac{1}{2}, \lrangle{\check{v}, \pi(\cdot) v}, \xi_0 \right), \\
		\check{Z}_\lambda \left( \check{\eta}_{\Psi}, v \otimes \check{v}, \xi \right) & = Z^{\mathrm{GJ}}\left( -\lambda + \frac{1}{2}, \lrangle{\check{\pi}(\cdot) \check{v}, v}, \xi_0 \right), \\
		\gamma(\sigma \boxtimes \check{\sigma}, \lambda)\left( \check{\eta}_{\Psi} \right) & = \gamma^{\mathrm{GJ}}\left( \lambda + \frac{1}{2}, \sigma \right) (\eta_\Omega) ,
	\end{align*}
	so that functional equation in Theorem \ref{prop:LFE} reduces to the usual Godement--Jacquet functional equation. We refer to \cite[\S 6.4]{LiLNM} for detailed explanations when $D$ is split; the general case is analogous.
\end{example}

\begin{remark}\label{rem:GJ-gamma-behavior}
	In the Godement--Jacquet case, the automorphism $\check{\eta} \mapsto m_a \circ \check{\eta}$ of $\mathcal{N}_\pi(\check{X}^+)$ (Proposition \ref{prop:gamma-symmetry}) is given by $\omega_\sigma(a)^{-1} \cdot \identity$, for all $a \in \R^\times$.
\end{remark}

\begin{remark}
	If a general additive character $\psi$ is used, one has to use the self-dual versions $\mathcal{F}^{\mathrm{sd}}_\psi$ and $\gamma^{\mathrm{sd}}$ (Remark \ref{rem:sd-version}) in the functional equation to regain compatibility with Godement--Jacquet.
\end{remark}

\subsection{The complex case}\label{sec:cplx-case}

The Hypothesis \ref{hyp:PVS} and the main Theorems in \S\ref{sec:main-thm} can all be formulated over $\CC$. The complex case can be reduced to the previous case over $\R$ as follows.

Let us write $\Res := \Res_{\CC | \R}$ for the functor of restriction of scalars à la Weil along $\CC | \R$, applied to $\CC$-varieties, etc. If a triplet $(G, \rho, X)$ over $\CC$ satisfies Hypothesis \ref{hyp:PVS}, so does $(\Res G, \Res \rho, \Res X)$; recall that $(\Res X)(\R) = X(\CC)$, $(\Res X^+)(\R) = X^+(\CC)$, $(\Res G)(\R) = G(\CC)$. There are canonical isomorphisms
\[\begin{tikzcd}
	\mathbf{X}^*(\Res G) \arrow[r, "\sim"] & \mathbf{X}^*(G) \\
	\mathbf{X}^*_{\Res \rho}(\Res G) \arrow[r, "\sim"] \arrow[phantom, u, "\subset" description, sloped] & \mathbf{X}^*_\rho(G) \arrow[phantom, u, "\subset" description, sloped] .
\end{tikzcd}\]

If $f \in \CC(X)$ is a relative invariant of eigencharacter $\omega \in \mathbf{X}^*_\rho(G)$, then its norm $f \cdot \overline{f} \in \R(\Res X)$ has the eigencharacter in $\mathbf{X}^*_{\Res \rho}(\Res G)$ corresponding to $\omega$.

Taking contragredient commutes with $\Res$, since the pairing $\Tr_{\CC|\R} \circ \lrangle{\cdot, \cdot}: X(\CC) \times \check{X}(\CC) \to \R$ is perfect. Fix an additive character $\psi$ of $\R$ and take the additive character $\psi_{\CC} := \psi \circ \Tr_{\CC|\R}$ of $\CC$. The Schwartz spaces for $X(\CC)$ and $(\Res X)(\R)$ can be identified, and so do the Fourier transforms.

Finally, an SAF representation $\pi$ of $G(\CC)$ is the same as an SAF representation of $(\Res G)(\R)$ on the same Fréchet space. The recipes above identifies $\mathcal{N}_\pi(X^+)$ (over $\CC$) with $\mathcal{N}_\pi(\Res X^+)$ (over $\R$). Hence the generalized matrix coefficients on $X^+(\CC)$ are the same as those on $(\Res X^+)(\R)$. All in all, the generalized zeta integral (Definition \ref{def:zeta}) for $(G, \rho, X)$ reduces immediately to the case for $(\Res G, \Res \rho, \Res X)$, and the theorems in \S\ref{sec:main-thm} carry over verbatim.

\section{Convergence and meromorphic continuation}\label{sec:conv-mero}
Throughout this section, we fix a triplet $(G, \rho, X)$ as in Hypothesis \ref{hyp:PVS} and an SAF representation $\pi$ of $G(\R)$.

\subsection{Proof of convergence}

Fix $\eta \in \mathcal{N}_\pi(X^+)$ and write $\eta = \eta_0 |\Omega|^{1/2}$ for some chosen $\Omega \in \topwedge \check{X} \smallsetminus \{0\}$.

The argument for Theorem \ref{prop:convergence} is the same as that of \cite[Theorem 6.6.4]{Li18}. It is based on the following result. Some terminologies from real algebraic geometry such as Nash functions will be needed; we refer to \cite[Chapter 8]{BCR98} for details.

\begin{proposition}\label{prop:soft-bound}
	There exist a continuous semi-norm $q: V_\pi \to \R_{\geq 0}$ and a Nash function $p: X^+(\R) \to \R_{\geq 0}$, such that
	\begin{equation*}
		|\eta_0(v)(x)| \leq q(v) p(x), \quad v \in V_\pi, \; x \in X^+(\R).
	\end{equation*}
\end{proposition}
\begin{proof}
	First, by \cite[Theorem 10.5]{Li19} there exists a ``weight function'' $w: X^+(\R) \to \R_{\geq 1}$ together with a continuous semi-norm $q: V_\pi \to \R_{\geq 0}$ such that
	\begin{compactitem}
		\item $w$ is continuous and subanalytic,
		\item $\{x \in X^+(\R): w(x) \leq B \}$ is compact for all $B > 0$,
		\item $|\eta_0(v)(x)| \leq w(x) q(v)$ for all $v \in V_\pi$ and $x \in X^+(\R)$.
	\end{compactitem}
	In fact, this can be deduced from the moderate growth of SAF representations. It is also deducible from the finer results in \cite{KKS17}.

	Next, we embed $X$ as an open dense subset of a smooth projective $\R$-variety $\overline{X}$ and let $\partial \overline{X} := \overline{X} \smallsetminus X^+$. By \cite[Proposition 7.5]{Li19}, $1/w$ extends uniquely to a subanalytic continuous function $\overline{X}(\R) \to \R_{\geq 0}$, still denoted as $1/w$, whose zero locus is exactly $\partial \overline{X}(\R)$.

	Note that $\overline{X}(\R)$ is affine in the \emph{real} sense; see \cite[Theorem 3.4.4]{BCR98}. Therefore $\partial \overline{X}(\R)$ is also the zero locus of some polynomial function $p_0: \overline{X}(\R) \to \R_{\geq 0}$; see \cite[Proposition 2.1.3]{BCR98}. We claim that there exist constants $a \in \Z_{\geq 1}$ and $C \in \R_{> 0}$ such that
	\[ p_0 \leq C \cdot (1/w)^a \quad \text{over}\; \overline{X}(\R). \]
	Indeed, this is due to the compactness of $\overline{X}(\R)$ and Łojasiewicz's inequality for subanalytic functions; see \cite[Theorem 6.4]{Li19}.

	All in all, we have
	\[ |\eta_0(v)(x)| \leq w(x) q(v) \leq \left( C p_0^{-1} \right)^{1/a} q(v), \quad x \in X^+(\R). \]
	Notice that $C p_0^{-1}$ and its $a$-th root are positive Nash functions on $X^+(\R)$. This completes the proof.
\end{proof}

The notations below are the same as those in Theorem \ref{prop:convergence}.

\begin{proof}[Proof of Theorem \ref{prop:convergence}]
	As in Definition \ref{def:zeta}, we write
	\[ Z_\lambda(\eta, v, \xi) = \int_{X(\R)} \eta_0(v) |f|^{\lambda - \lambda_0} \xi_0 |\Omega| \]
	where $\xi_0 \in \Schw_0(X)$.
	
	By \cite[Lemma 6.6.5]{Li18}, whose proof applies under our Hypothesis \ref{hyp:PVS}, there exist $\mu \in \Lambda_{\R}$ and a Nash function $p_1$ on $X(\R)$ such that $|f|^\mu p \leq p_1$. Hence
	\[ |f|^{\Re(\lambda) - \lambda_0} p \leq |f|^{\Re(\lambda) - \lambda_0 - \mu} p_1. \]
	Therefore, Proposition \ref{prop:soft-bound} implies
	\begin{equation}\label{eqn:conv-estimate}
		\left| \eta_0(v) |f|^{\lambda - \lambda_0} \xi_0 \right| \leq q(v) |f|^{\Re(\lambda) - \lambda_0 - \mu} p_1 |\xi_0|.
	\end{equation}
	
	We claim that, when $\Re(\lambda)$ is constrained in some compact subset $C \subset \Lambda_{\R}$ satisfying $\theta \relgeq{X} \lambda_0 + \mu$ for all $\theta \in C$, the function $|f|^{\Re(\lambda) - \lambda_0 - \mu}$ is uniformly bounded by a polynomial function. To see this, take $\lambda^* = \sum_{i=1}^r 2\lambda^*_i \otimes \omega_i \in 2\Lambda_{\Z}$ such that $\lambda^* \relgeq{X} \theta - \lambda_0 - \mu$ for all $\theta \in C$. Then
	\[ |f|^{\Re(\lambda) - \lambda_0 - \mu} \leq \prod_{i=1}^r \left( 1 + |f_i|^{2\lambda^*_i} \right). \]
	
	On the other hand, $p_1$ is Nash over $X(\R)$, and by \cite[Proposition 2.6.2]{BCR98} every Nash function on $X(\R)$ is bounded by some polynomial. It follows that the right hand side of \eqref{eqn:conv-estimate} is integrable over $X(\R)$ after multiplied by $|\Omega|$, when $\Re(\lambda) \relgg{X} 0$. The implied lower bound $\kappa$ of $\Re(\lambda)$ depends only on $\pi$ and $\eta$. As $\dim_{\CC} \mathcal{N}_\pi(X^+)$ is finite, $\kappa$ can even be made uniform in $\eta$.

	When $(v, \xi)$ is fixed, the holomorphy in $\lambda$ and the boundedness in vertical strips in the range of convergence follow from \eqref{eqn:conv-estimate} and the bound on $|f|^{\Re(\lambda) - \lambda_0 - \mu}$ just obtained.

	In view of the topology on $\Schw(X)$, the continuity of $Z_\lambda(\eta, v, \xi)$ in $\xi$ follows easily from \eqref{eqn:conv-estimate}. The continuity in $v$ follows from that of $q(\cdot)$. Since $V_\pi$ and $\Schw(X)$ are Fréchet, joint continuity in $(v, \xi)$ follows.
\end{proof}

\subsection{Proof of meromorphic continuation}
Consider the sheaves $\mathscr{D}_{X^+}$ on $X^+$ and $\mathscr{D}_{X^+_{\CC}}$ on $X^+_{\CC}$ (recall \S\ref{sec:conventions}). Any distribution $u$ on $X^+(\R)$ generates a $\mathscr{D}_{X^+}$-module $\mathscr{D}_{X^+} \cdot u$, which complexifies into $\mathscr{D}_{X^+_{\CC}} \cdot u$ on $X^+_{\CC}$. We refer to \cite{Li18, Li19} for more a more detailed review of algebraic $\mathscr{D}$-modules, including especially the notion of holonomicity.

Fix a maximal compact subgroup $K \subset G(\R)$.

\begin{proposition}\label{prop:holonomicity}
	Let $\eta \in \mathcal{N}_\pi(X^+)$, written as $\eta = \eta_0 |\Omega|^{1/2}$ for some $\Omega \in \topwedge \check{X} \smallsetminus \{0\}$, and let $v \in V_\pi^{K\text{-fini}}$. Then $\mathscr{D}_{X^+_{\CC}} \cdot \eta_0(v)$ is a holonomic $\mathscr{D}_{X^+_{\CC}}$-module on $X^+_{\CC}$.
\end{proposition}
\begin{proof}
	Since $\eta_0(v)$ is $K$-finite and $\mathcal{Z}(\mathfrak{g})$-finite, whilst $X^+_{\CC}$ is a spherical homogeneous space, the holonomicity is assured by \cite[Proposition 10.4]{Li19}. As remarked in \textit{loc.\ cit.}, this can also be proved via the arguments of \cite{AGM16}.
\end{proof}

The notations below are the same as those in Theorem \ref{prop:meromorphy}.

\begin{proof}[Proof of Theorem \ref{prop:meromorphy}]
	The argument for meromorphic continuation is exactly the same as \cite[Theorems 6.8.2, 6.8.4]{Li18}. It proceeds in two stages.

	First, for $v \in V_\pi^{K\text{-fini}}$, one employs the method of Bernstein--Sato $b$-functions as explained in \cite[Appendice]{BD92}. The sole input here is the holonomicity established in Proposition \ref{prop:holonomicity}. This step corresponds to \cite[Theorem 6.8.2]{Li18}; it also produces the holomorphic function $L(\eta, \lambda)$.
	
	Secondly, let $\mathcal{S}(G)$ be the algebra of Schwartz measures on $G(\R)$, which acts on $\Schw(X)$ and also on any SAF representation of $G(\R)$. One uses $V_\pi = \pi(\mathcal{S}(G)) V_\pi^{K\text{-fini}}$ to treat general $v \in V_\pi$. This corresponds to \cite[Theorem 6.8.4]{Li18}, the main analytic device being the Gelfand--Shilov principle of \cite[Proposition 6.8.3]{Li18}. Specifically, the recipe is
	\[ v = \sum_{i=1}^m \pi(\Xi_i) v_i \implies LZ_\lambda \left(\eta, v, \xi  \right) := \sum_{i=1}^m LZ_\lambda\left( \eta, v_i, \check{\Xi}_i \xi \right) \]
	where $v_i \in V_\pi^{K\text{-fini}}$, $\Xi_i \in \mathcal{S}(G)$ and $\check{\Xi}_i(g) = \Xi_i(g^{-1})$. In \textit{loc.\ cit.}, this is shown to be well-defined and compatible with the original definition in the range of convergence.

	The joint continuity and $G(\R)$-invariance of the bilinear form $LZ_\lambda(\eta, \cdot, \cdot)$ are also established in \textit{loc.\ cit.}
\end{proof}

\begin{corollary}\label{prop:zeta-translation}
	Let $\lambda \in \Lambda_{\CC}$, $\xi \in \Schw(X)$.
	\begin{enumerate}[(i)]
		\item Define $\eta \mapsto \eta_\lambda$ as in \eqref{eqn:eta-lambda}, then
		\[ Z_{\mu + \lambda}(\eta, v, \xi) = Z_\mu\left( \eta_\lambda, v, \xi \right) \]
		as meromorphic families in $\mu \in \Lambda_{\CC}$.
		\item Suppose that $\lambda \in \Lambda_{\Z}$ and $h \in \R[X]$ is a relative invariant of eigencharacter $\lambda$ satisfying $h = |f|^\lambda$ on $X(\R)$, then
		\[ Z_{\mu + \lambda}(\eta, v, \xi) = Z_\mu\left( \eta, v, h \xi \right). \]
	\end{enumerate}
\end{corollary}
\begin{proof}
	For both (i) and (ii), we begin with the case that $\mu$ and $\lambda + \mu$ are both in the range of convergence for zeta integrals. Then the equalities hold by Definition \ref{def:zeta}. The general case follows by meromorphic continuation.
\end{proof}

\subsection{Order of tempered distributions}
The results below will be applied to prove the functional equation.

Choose a basis for the $\R$-vector space $X$. This gives rise to the dual basis of $\check{X}$, the standard volume form $\Omega$ and the standard norm $\|\cdot\|$ on $X(\R) \simeq \R^n$. The elements of $\Schw(X)^\vee \simeq \Schw_0(X)^\vee$ can be viewed as \emph{tempered distributions} on $X(\R)$.

Recall that the topology on $\Schw_0(X)$ is determined by the semi-norms
\[ \|\xi_0\|_{a,b} := \sup_{\substack{|\underline{\alpha}| \leq a \\ |\underline{\beta}| \leq b}} \; \sup_{x \in X(\R)} \left| x^{\underline{\beta}} \cdot \partial^{\underline{\alpha}} \xi(x) \right| , \quad a,b \in \Z_{\geq 0} \]
where $\underline{\alpha} = (\alpha_1, \ldots, \alpha_n)$, $\underline{\beta} = (\beta_1, \ldots, \beta_n)$, $|\alpha| := \sum_i \alpha_i$, and $\partial^{\underline{\alpha}} = \partial_1^{\alpha_1} \cdots \partial_n^{a_n}$, $x^{\underline{\beta}} = x_1^{b_1} \cdots x_n^{b_n}$ are the standard terminologies of multi-indices.

\begin{definition}
	Consider $(a, b) \in \Z_{\geq 0}^2$. We say that a tempered distribution $Z$ on $X$ \emph{has order $\leq (a,b)$} if there exists a constant $C > 0$ such that $|Z(\xi_0)| \leq C \|\xi_0\|_{(a,b)}$ for all $\xi_0 \in \Schw_0(X)$.
\end{definition}

Some basic facts:
\begin{itemize}
	\item Every tempered distribution has order $\leq (a, b)$ for sufficiently large $a, b$.
	\item If $a' \geq a$ and $b' \geq b$ then having order $\leq (a, b)$ implies having order $\leq (a', b')$.
\end{itemize}

The following is also well-known.

\begin{proposition}[See for example {\cite[Theorem 25.1]{Tr67}} and its proof]\label{prop:order-Fourier}
	Suppose that $\check{Z} \in \Schw(\check{X})^\vee$ has order $\leq (a, b)$. Then $\check{Z} \circ \mathcal{F} \in \Schw(X)^\vee$ has order $\leq (b, a+n+1)$.
\end{proposition}

Next, consider the tempered distributions $Z_\lambda(\eta, v, \cdot)$ for $\lambda$ in the range of convergence.

\begin{proposition}
	Let $\eta \in \mathcal{N}_\pi(X^+)$ and $v \in V_\pi$. For $\kappa$ as in Theorem \ref{prop:convergence}, there exists $b \in \Z_{\geq 0}$ such that $Z_\lambda(\eta, v, \cdot)$ has order $\leq (0, b)$ for all $v \in V_\pi$ and all $\lambda \in \Lambda_{\CC}$ with $\Re(\lambda) \relgeq{X} \kappa$.
\end{proposition}
\begin{proof}
	This stems from the estimate \eqref{eqn:conv-estimate} in the proof of Theorem \ref{prop:convergence} and the subsequent discussions: with the notations therein, it suffices to take an even $b$ such that
	\[ |f(x)|^{\Re(\lambda) -\lambda_0 - \mu} p_1(x) \leq C (1 + \|x\|^2)^{b/2} \]
	for all $x \in X(\R)$, where $C$ is some constant.
\end{proof}

For the next proposition, take a holomorphic function $L(\eta, \lambda)$ and define $LZ_\lambda(\eta, \cdot, \cdot)$ as in Theorem \ref{prop:meromorphy}. We shall also fix a maximal compact subgroup $K \subset G(\R)$.

\begin{proposition}\label{prop:order-estimate}
	Let $\eta \in \mathcal{N}_\pi(X^+)$ and $v \in V_\pi^{K\text{-fini}}$. For every $c \in \Lambda_{\R}$, there exists $(a, b) \in \Z_{\geq 0}^2$ such that $LZ_\lambda(\eta, v, \cdot)$ has order $\leq (a, b)$ for all $\lambda \in \Lambda_{\CC}$ with $\Re(\lambda) \relgeq{X} c$.
\end{proposition}
\begin{proof}
	Recall from \cite[Proposition 6.4.3]{Li18} or \cite[Proposition A.1]{BD92} that the meromorphic continuation of $Z_\lambda(\eta, v, \cdot)$ is achieved by applying certain algebraic differential operators to $\eta(v)|f|^\lambda$, with the effect of shifting the domain of $Z_\lambda(\eta, v, \cdot)$ leftwards, and possibly creates poles. Starting from $\{\lambda: \Re(\lambda) \relgeq{X} \kappa \}$ on which $Z_\lambda(\eta, v, \cdot)$ has order $\leq (0, b')$ for some $b' \geq 0$, one covers $\{\lambda: \Re(\lambda) \relgeq{X} c \}$ after a finite number of such shifts. This procedure increases the order by some pair of positive integers. Our assertion follows.
\end{proof}

\section{Invariant differential operators}\label{sec:diff-op}

\subsection{On certain Capelli operators}

In this subsection, we take $G$ to be a connected reductive $\CC$-group, and the algebraic varieties are taken over $\CC$. For any smooth $G$-variety $Z$, there is a natural homomorphism $\mathcal{U}(\mathfrak{g}) \to \mathcal{D}(Z)$ which restricts to $\mathcal{Z}(\mathfrak{g}) \to \mathcal{D}(Z)^G$.

Now consider a finite-dimensional $\CC$-vector space $X$ with a right $G$-action, given by a homomorphism $\rho: G \to \GL(X)$ between algebraic groups.

\begin{definition}\label{def:MF-space}
	For $X$ as above, we say $X$ is \emph{multiplicity-free} if $X$ is spherical as a $G$-variety. Equivalently, the left $G$-module $\CC[X]$ decomposes with multiplicity one.
\end{definition}

Being multiplicity-free implies the existence of an open dense $G$-orbit $X^+ \subset X$, hence $(G, \rho, X)$ is a prehomogeneous vector space. In contrast with Hypothesis \ref{hyp:PVS}, here $X^+$ is not necessarily affine and $(G, \rho, X)$ is not necessarily regular.

\begin{lemma}\label{prop:auto-regularity}
	For a multiplicity-free $G$-space $X$, we have $\mathcal{D}(X^+)^G = \mathcal{D}(X)^G$. Moreover $\mathcal{D}(X)^G$ is commutative and finitely generated as a $\mathcal{Z}(\mathfrak{g})$-module.
\end{lemma}
\begin{proof}
	We recall from \cite[p.271]{Kn94} that Knop defined an algebra $\mathfrak{Z}(Z) := \mathfrak{U}(Z)^G$ for any smooth $G$-variety $Z$, where $\mathfrak{U}(Z) \subset \mathcal{D}(Z)$ is the subalgebra of ``completely regular differential operators''. A key fact in \cite[p.262]{Kn94} is that $\mathfrak{U}(Z)$ is a birational-equivariant invariant of $Z$, hence so is $\mathfrak{Z}(Z)$; furthermore, $\mathfrak{Z}(Z) = \mathcal{D}(Z)^G$ when $Z$ is spherical, by \cite[pp.254--255]{Kn94}.
	
	Applying these results to $Z \in \left\{ X, X^+ \right\}$, we conclude that $\mathcal{D}(X^+)^G = \mathfrak{Z}(X^+) = \mathfrak{Z}(X) = \mathcal{D}(X^G)$. The commutativity of $\mathcal{D}(X)^G$ and finite generation over $\mathcal{Z}(\mathfrak{g})$ are included in the main Theorem in \textit{loc.\ cit.}
\end{proof}

Fix a Borel subgroup $B \subset G$ and set $T := B/B_{\mathrm{der}}$. Let $X$ be a multiplicity-free $G$-space. There are decompositions
\begin{equation}\label{eqn:P-D} \begin{aligned}
	\CC[X] & = \bigoplus_{\lambda \in \mathbf{X}^*(T)^+} \mathcal{P}_\lambda, \\
	\CC[\check{X}] & = \bigoplus_{\lambda \in \mathbf{X}^*(T)^+} \mathcal{D}_\lambda,
\end{aligned}\end{equation}
where
\begin{compactitem}
	\item $\mathbf{X}^*(T)^+$ denotes the set of dominant weights in $\mathbf{X}^*(T)$,
	\item $\mathcal{P}_\lambda$ is the simple $G$-submodule with lowest weight $-\lambda$, occurring with multiplicity $\leq 1$,
\end{compactitem}
and the decomposition of $\CC[\check{X}]$ is obtained from that of $\CC[X]$ by duality; in particular $\mathcal{D}_\lambda$ is the simple $G$-submodule with highest weight $\lambda$, with multiplicity $\leq 1$. Thus $\check{X}$ is also a multiplicity-free $G$-space.

Note that in \cite{Kn98}, $G$ acts on the left of $X$. One switches between left and right actions on $X$ by $g^{-1} x = xg$, and the left $G$-module $\CC[X]$ remains unaffected. Ditto for $\CC[\check{X}]$.

\begin{theorem}[see {\cite{HU91}} or {\cite{Kn98}}]\label{prop:Capelli}
	For a multiplicity-free $G$-space $X$, we have an isomorphism of $\CC$-algebras
	\[ \begin{tikzcd}[row sep=tiny]
		C: \CC[X \times \check{X}]^G \arrow[equal, r] & \displaystyle\bigoplus_{\lambda \in \mathbf{X}^*(T)^+} \left( \mathcal{P}_\lambda \dotimes{\CC} \mathcal{D}_\lambda \right)^G \arrow[r, "\sim"] & \mathcal{D}(X)^G \\
		& \sum_\lambda p_\lambda \otimes q_\lambda \arrow[mapsto, r] &  \sum_\lambda p_\lambda q_\lambda
	\end{tikzcd}\]
	where we regard $p_\lambda \in \CC[X]$ and $q_\lambda \in \CC[\check{X}]$ as algebraic differential operators on $X$.
\end{theorem}

Notice that if $G$, $X$ descend to a subfield of $\CC$, so does $C$.

We will mainly use the terms with $\dim \mathcal{P}_\lambda = 1 = \dim \mathcal{D}_\lambda$. In other words, we consider relative invariants $p \in \CC[X]$ and $q \in \CC[\check{X}]$ with opposite eigencharacters $-\lambda$ and $\lambda$. Then $p \otimes q$ are automatically $G$-invariant and $C(p \otimes q)$ is an instance of the \emph{Capelli operators} introduced in \cite{HU91}.

\subsection{Knop's Harish-Chandra isomorphism}
For any smooth $G$-variety $X$, Knop \cite[Theorem 6.5]{Kn94} has defined a Harish-Chandra isomorphism which realizes $\mathfrak{Z}(X)$ as the coordinate algebra of some explicitly defined variety. Below we review the simpler case of multiplicity-free spaces, following \cite{Kn98}.

Fix a multiplicity-free $G$-space $X$ (Definition \ref{def:MF-space}) over $\CC$, with open $G$-orbit $X^+$. Fix a Borel subgroup $B \subset G$ and let $T := B/B_{\mathrm{der}}$. Let $W$ be the corresponding abstract Weyl group acting on $T$; see \cite[p.137]{CG10}. Define
\begin{itemize}
	\item $\Lambda(X)^+ := \left\{ \lambda \in \subset \mathbf{X}^*(T): \mathcal{P}_\lambda \neq 0 \right\}$ where $\mathcal{P}_\lambda$ is as in \eqref{eqn:P-D}, and let $\Lambda(X) \subset \mathbf{X}^*(T)$ be the subgroup generated by $\Lambda(X)^+$;
	\item $\mathfrak{a}^*_X := \Lambda(X) \otimes \R$, which is a subspace of $\mathfrak{a} := \mathbf{X}^*(T) \otimes \R$;
	\item $\rho := \frac{1}{2} \sum \alpha \; \in \mathfrak{a}^*$ where $\alpha$ ranges over the positive roots;
	\item $W_X$: the little Weyl group of $X^+$, which is a reflection group acting on $\mathfrak{a}_\mathbf{X}^*$ and embeds into the normalizer $N_W\left( \rho + \mathfrak{a}_\mathbf{X}^* \right)$.
\end{itemize}
By \cite[3.2]{Kn98}, the submonoid $\Lambda(X)^+$ in $\mathbf{X}^*(T)$ is generated by linearly independent elements $\chi_1, \ldots, \chi_m$ in $\mathbf{X}^*(T)$, and its $\Z$-span $\Lambda(X)$ is the group of all weights of $B$-eigenfunctions in $\CC(\check{X})$. We refer to \textit{loc.\ cit.} for further details.

For all $D \in \mathcal{D}(X)^G$ and $\lambda \in \Lambda(X)$, the action of $D$ on $\mathcal{P}_\lambda$ must be a scalar, say $c_D(\lambda)$, due to multiplicity-freeness. We obtain a map $\mathcal{D}(X)^G \to \mathrm{Maps}(\Lambda(X)^+, \CC)$, mapping $D$ to $c_D$.

\begin{theorem}[F.\ Knop {\cite[4.8]{Kn98}}]\label{prop:Knop-HC}
	For every $D \in \mathcal{D}(X)^G$, the function $c_D$ extends uniquely to a polynomial $c_D \in \CC[\mathfrak{a}_\mathbf{X}^*]$, and the map
	\[\begin{tikzcd}[row sep=tiny]
		\mathrm{HC}: \mathcal{D}(X)^G \arrow[r] & \CC\left[ \rho + \mathfrak{a}_\mathbf{X}^* \right] \\
		D \arrow[mapsto, r] & {[ x \mapsto c_D(x - \rho) ]}
	\end{tikzcd}\]
	is an injective homomorphism of $\CC$-algebras, with image equal to $\CC[\rho + \mathfrak{a}_\mathbf{X}^*]^{W_X}$ or equivalently $\CC[\left( \rho + \mathfrak{a}_\mathbf{X}^* \right) /\!/ W_X]$, where $/\!/$ denotes the categorical quotient.
\end{theorem}

\begin{remark}\label{rem:inf-char}
	Whenever $\mathcal{D}(X)^G$ acts on some $\CC$-vector space $V$ and $v \in V$ is a joint generalized eigenvector therein, we may attach an \emph{infinitesimal character} $\chi_v$ to $v$; it is an element of $\left( \rho + \mathfrak{a}_\mathbf{X}^* \right) /\!/ W_X$. Specifically, for all $D \in \mathcal{D}(X)^G$ we have
	\[ N \gg 0 \implies \left(D - \mathrm{HC}(D)(\chi_v) \cdot \identity_V \right)^N v = 0. \]
\end{remark}

As in \S\ref{sec:relative-invariants}, define $\mathbf{X}^*_\rho(G) \subset \mathbf{X}^*(G)$ to be the lattice of eigencharacters of relative invariants. In fact, $\mathbf{X}^*_\rho(G) \subset \Lambda(X)^{W_X}$. Every relative invariant can be viewed as an algebraic differential operator on $X^+$ of order zero.

\begin{remark}
	Let $\lambda \in \mathbf{X}^*_\rho(G) \otimes \CC$. The translation $x \mapsto x + \lambda$ makes sense on $(\rho + \mathfrak{a}_\mathbf{X}^*) /\!/ W_X$ since $\lambda$ is $W_X$-invariant; in fact $\lambda$ is even $W$-invariant.
\end{remark}

\begin{proposition}[Algebraic twists]\label{prop:HC-twist}
	Let $f \in \CC(X)$ be a relative character of eigencharacter $\lambda \in \mathbf{X}^*_\rho(G)$. For all $D \in \mathcal{D}(X)^G$ and $s \in \Z$, the differential operator $D_{f, s} := f^{-s} \circ D \circ f^s \in \mathcal{D}(X^+)$ belongs to $\mathcal{D}(X)^G$; furthermore,
	\[ \mathrm{HC}(D_{f, s})(x) = \mathrm{HC}(D)(x - s\lambda), \quad x \in (\rho + \mathfrak{a}_\mathbf{X}^*) /\!/ W_X . \]
\end{proposition}
\begin{proof}
	Clearly $D_{f, s}$ is $G$-invariant. It extends to $X$ by Lemma \ref{prop:auto-regularity}. For the remaining assertion, we have to compare $c_D$ and $c_{D_{f, s}}$. Let $\mu \in \Lambda(X)^+$ be ``sufficiently positive'' (see below), and let $h \in \mathcal{P}_\mu$ be a corresponding element of lowest weight $-\mu$. We have
	\[\begin{tikzcd}[row sep=tiny]
		h \arrow[mapsto, r, "f^s"] & hf^s \arrow[mapsto, r, "D"] & c_D(\mu - s\lambda) hf^s \arrow[mapsto, r, "f^{-s}"] & c_D(\mu - s\lambda) h = D_{f, s} h \\
		-\mu & -\mu + s\lambda & -\mu + s\lambda & -\mu
	\end{tikzcd}\]
	where the second row indicates the weights; notice that the functions in the first row are all lowest weight vectors. Here we assume that $\mu, \mu - s\lambda \in \Lambda(X)^+$. For such $\mu$, we infer that
	\[ c_{D_{f, s}}(\mu) = c_D(\mu - s\lambda). \]
	As $\Lambda(X)$ is a full-rank lattice in $\mathfrak{a}_\mathbf{X}^*$, it is then elementary to conclude that $c_{D_{f, s}}(x) = c_D(x - s\lambda)$ for all $x \in \mathfrak{a}_\mathbf{X}^*$.
\end{proof}

\begin{proposition}\label{prop:top-nonvanishing}
	Denote by $c_D^{\mathrm{top}}$ the top homogeneous component of $c_D \in \CC[\mathfrak{a}^*]$, for every $D \in \mathcal{D}(X)^G$. Consider the data
	\begin{itemize}
		\item $f \in \CC[X]$, $\check{f} \in \CC[\check{X}]$: polynomial relative invariants with opposite eigencharacters;
		\item $D := C(f \otimes \check{f}) \in \mathcal{D}(X)^G$;
		\item $\mu \in \mathbf{X}^*_\rho(G) \subset \mathfrak{a}_\mathbf{X}^*$: the eigencharacter of some non-degenerate relative invariant $h \in \CC[X]$ for $(G, \rho, X)$ (recall \S\ref{sec:relative-invariants}).
	\end{itemize}
	Then $c_D^{\mathrm{top}}(-\mu) \neq 0$.
\end{proposition}
\begin{proof}
	More generally, consider a homogeneous element $E \in (\mathcal{P}_\lambda \otimes \mathcal{D}_\lambda)^G$ with $D := C(E)$. By \cite[4.5]{Kn98}, $c_D^{\mathrm{top}}$ equals $\overline{c}(E)$ where
	\[ \overline{c}: \left( \CC[X] \dotimes{\CC} \CC[\check{X}] \right)^G \to \CC[\mathfrak{a}_\mathbf{X}^*] \]
	is defined as follows. Let $\mathring{X} \subset X^+$ be the open $B$-orbit. There is a well-defined map $\mathfrak{a}_\mathbf{X}^* \times \mathring{X} \xrightarrow{\phi} \check{X}$ given by
	\[ (\chi, v) = \left( \sum_i a_i \lambda_i, v \right) \longmapsto \sum_i a_i \underbracket{\left( f_i^{-1} \dd f_i \right)}_{\text{regular on $\mathring{X}$}}(v) =: \phi_\chi(v), \]
	where $a_i \in \CC$ and $\lambda_i \in \Lambda(X)^+$, with a $B$-eigenfunction $f_i \in \mathcal{P}_{\lambda_i}$; so $\phi$ is linear in $\chi$. For each $v \in \mathring{X}$, set
	\[ \mathfrak{a}_\mathbf{X}^*(v) := \left\{ (v, \phi_\chi(v)) \in X \times \check{X} : \chi \in \mathfrak{a}_\mathbf{X}^* \right\}. \]
	
	By \cite[p.307]{Kn98}, $\chi \mapsto (v, \phi_\chi(v))$ defines an isomorphism from $\mathfrak{a}_\mathbf{X}^*$ onto the affine subspace $\mathfrak{a}_\mathbf{X}^*(v) \subset X \times \check{X}$. Now we put
	\[ \overline{c}(E) := E |_{\mathfrak{a}_\mathbf{X}^*(v)}, \; \text{identified as an element of } \CC[\mathfrak{a}^*_X]. \]

	Next, take $E := f \otimes \check{f}$, noting that relative invariants are homogeneous. It remains to prove that $\overline{c}(E)(-\mu) \neq 0$.

	Let $h \in \CC[X]$ be a non-degenerate relative invariant of eigencharacter $\mu$, so $h \in \mathcal{P}_{-\mu}$. Take $\chi = -\mu$ in the construction above to see that for any $v \in \mathring{X}$,
	\[ \overline{c}(E)(-\mu) = (f \otimes \check{f})(v, \phi_\chi(v)) = f(v) \check{f}(\phi_\chi(v)) \neq 0 \]
	by the non-vanishing of relative invariants on $X^+$ and $\check{X}^+$, since $\phi_\chi(v) = (h^{-1} \dd h)(v) \in \check{X}^+$ by the non-degeneracy of $h$.
\end{proof}

\subsection{Analytic twists}
In this subsection we work primarily over $\R$. Let $(G, \rho, X)$ be as in Hypothesis \ref{hyp:PVS}. Take a pair $f \in \R[X]$, $\check{f} \in \R[\check{X}]$ of relative invariants as furnished by Corollary \ref{prop:rel-invariant-pair}. Theorem \ref{prop:Capelli} then affords us an invariant algebraic differential operator $C(f \otimes \check{f})$ on $X$.

On the other hand, for every $\lambda \in \Lambda_{\CC}$, we view $|f|^\lambda$ as a real-analytic differential operator of order zero on $X^+(\R)$. It makes sense to define the invariant real-analytic differential operator
\[ C_\lambda(f \otimes \check{f}) := |f|^{-\lambda} \circ C(f \otimes \check{f}) \circ |f|^\lambda, \quad \lambda \in \Lambda_{\CC} \]
on $X^+(\R)$. They act on $\mathcal{N}_\pi(X^+)$ in view of Definition \ref{def:D-action-N}, for every SAF representation $\pi$ of $G(\R)$.

\begin{remark}\label{rem:analytic-twist}
	Suppose that $h \in \R(X)$ is a relative invariant with eigencharacter $\mu \in \mathbf{X}^*_\rho(G)$, such that $h > 0$ and $h = |f|^\mu$ on $X^+(\R)$. The analytic twist $C_{s\mu}(f \otimes \check{f})$ then comes from the algebraic twist $C(f \otimes \check{f})_{h, s} \in \mathcal{D}(X)^G$ (Proposition \ref{prop:HC-twist}) for all $s \in \Z$.
\end{remark}

\begin{proposition}\label{prop:Capelli-holomorphy}
	Let $\pi$ be an SAF representation. For every $\eta \in \mathcal{N}_\pi(X^+)$, the family $C_\lambda(f \otimes \check{f})(\eta)$ inside $\mathcal{N}_\pi(X^+)$ is holomorphic in $\lambda \in \Lambda_{\CC}$.
\end{proposition}
\begin{proof}
	The argument is a variant of that for Lemma \ref{prop:holomorphy-test}. For each $x \in X^+(\R)$, consider the evaluation map $\mathrm{ev}_x: C^\infty(X^+) \rightiso C^\infty(X^+; \R) \to \CC$ at $x$, where the first isomorphism comes from the trivialization of $\mathcal{L}^{1/2}$ in Lemma \ref{prop:avoid-density}. These maps are continuous and $\bigcap_x \Ker(\mathrm{ev}_x) = \{0\}$.
	
	Now consider the linear functionals $\eta \mapsto \mathrm{ev}_x(\eta(v))$ of $\mathcal{N}_\pi(X^+)$ where $(x,v) \in X^+(\R) \times V_\pi$. Their kernel have trivial intersection, hence they generate the dual of $\mathcal{N}_\pi(X^+)$. Thus there exists a finite subset $F \subset X^+(\R) \times V_\pi$ such that
	\begin{align*}
		\mathcal{N}_\pi(X^+) & \hookrightarrow \CC^F \\
		\eta & \mapsto (\mathrm{ev}_x(\eta(v)))_{(x, v) \in F}.
	\end{align*}
	It suffices to show that for each $(x,v) \in F$, the function $\mathrm{ev}_x\left( C_\lambda(f \otimes \check{f}) \eta(v) \right)$ is holomorphic in $\lambda \in \Lambda_{\CC}$. This is obvious by unwinding various definitions.
\end{proof}

\section{Functional equation}\label{sec:FE}
Throughout this section, $(G, \rho, X)$ is as in Hypothesis \ref{hyp:PVS}. The SAF representation $\pi$ of $G(\R)$ is assumed to have a central character. We also fix a maximal compact subgroup $K \subset G(\R)$.

The notations for $Z_\lambda$, $\check{Z}_\lambda$, etc.\ are as in \S\ref{sec:desiderata}. In particular, the range of convergence for $Z_\lambda$ is given by $\Re(\lambda) \relgg{X} \kappa$ where $\kappa$ is as in Theorem \ref{prop:convergence}.

\subsection{A decomposition}
Fix basic relative invariants $f_1, \ldots, f_r \in \R[X]$ for $(G, \rho, X)$, with eigencharacters $\omega_1, \ldots, \omega_r$.

Let $A_G \subset G$ be the maximal split central torus, and let $A_G(\R)^\circ$ be the identity connected component of $A_G(\R)$. On the other hand, let $\mathfrak{a}_G := \Hom(\mathbf{X}^*(G), \R)$ and $H_G: G(\R) \to \mathfrak{a}_G$ be the Harish-Chandra homomorphism characterized by $\lrangle{\chi, H_G(g)} = |\chi(g)|$ for all $\chi \in \mathbf{X}^*(G)$. Set $G(\R)^1 := \Ker(H_G)$. It is well-known that $H_G: A_G(\R)^\circ \rightiso \mathfrak{a}_G$, and multiplication induces an isomorphism of Lie groups
\[ A_G(\R)^\circ \times G(\R)^1 \rightiso G(\R). \]

Define
\begin{align*}
	G(\R)_\rho & := \left\{ g \in G(\R) : \forall \chi \in \mathbf{X}^*_\rho(G), \; |\chi(g)| = 1 \right\}, \\
	X^+(\R)_\rho & := \left\{ x \in X^+(\R) : \forall 1 \leq i \leq r, \; |f_i(x)|=1 \right\}.
\end{align*}
Therefore $G(\R)_\rho \supset G(\R)^1$ and $G(\R)_\rho$ acts on the right of $X^+(\R)_\rho$.

Note that $\mathfrak{a}_\rho := \Hom(\mathbf{X}^*_\rho(G), \R)$ is a quotient of $\mathfrak{a}_G$. We can and do choose a splitting to realize $\mathfrak{a}_\rho$ as a direct summand of $\mathfrak{a}_G$, and let $A_\rho := H_G^{-1}(\mathfrak{a}_\rho) \subset A_G(\R)^\circ$. Note that $(|\omega_1|, \ldots, |\omega_r|)$ induces $A_\rho \rightiso (\R^\times_{>0})^r$.

\begin{proposition}\label{prop:rho-decomposition}
	With the choices above, we have real-analytic isomorphisms
	\[\begin{tikzcd}[row sep=tiny]
	G(\R)_\rho \times A_\rho \arrow[r, "\sim"] & G(\R) \\
	(g, a) \arrow[mapsto, r] & ga
	\end{tikzcd}\]
	and
	\[\begin{tikzcd}[row sep=tiny]
	X^+(\R)_\rho \times A_\rho \arrow[r, "\sim"] & X^+(\R) \\
	(x, a) \arrow[mapsto, r] & xa \\
	\left( y r(y)^{-1}, r(y) \right) & y \arrow[mapsto, l]
	\end{tikzcd}\]
	where $r: X^+(\R) \to A_\rho$ is the map characterized by $|f_i(y)| = |\omega_i(r(y))|$ for all $1 \leq i \leq r$.
\end{proposition}
\begin{proof}
	The decomposition of $G(\R)$ is routine to verify. As for the decomposition of $X^+(\R)$, one observes that $r(xa) = a$ for $(x, a) \in X^+(\R)_\rho \times A_\rho$; it follows readily that the two maps are mutually inverse.
\end{proof}

\begin{example}
	In the Godement--Jacquet case (Example \ref{eg:GJ}), we have
	\begin{align*}
		G(\R)_\rho & = \left\{ (g,h) \in D^\times \times D^\times: |\mathrm{Nrd}(g)| = |\mathrm{Nrd}(h)| \right\}, \\
		X^+(\R)_\rho & = \left\{x \in D(\R): |\mathrm{Nrd}(x)| = 1 \right\}.
	\end{align*}
	Note that $A_G(\R) \simeq \R^\times \times \R^\times$ and $\mathfrak{a}_G \simeq \R^2$ canonically. Take the splitting $\mathfrak{a}_\rho \hookrightarrow \mathfrak{a}_G$ so that
	\[ \mathfrak{a}_\rho := \R \times \{0\}, \quad A_\rho := \left\{ (a, 1): a \in \R^\times_{> 0} \right\}. \]
	Pick $\mathrm{Nrd}$ to be the basic relative invariant. Then the map $r: X^+(\R) \to A_\rho$ above is simply $y \mapsto \left( |\mathrm{Nrd}(y)|^{1/n}, 1 \right)$. The decompositions in Proposition \ref{prop:rho-decomposition} are then evident.
\end{example}

Let $C^\infty(X^+_\rho)$ stand for the Fréchet space of $C^\infty$ half-densities over $X^+(\R)_\rho$, which is a smooth $G(\R)_\rho$-representation. Notice that $A_\rho$ is isomorphic to the vector space $\R^r$ as Lie groups, hence there exists an invariant half-density $\ell \neq 0$ on $A_\rho$.

\begin{proposition}\label{prop:N-restriction}
	Let $\pi$ be an SAF representation of $G(\R)$ with central character. Choose any invariant half-density $\ell \neq 0$ on $A_\rho$. We have an isomorphism of $\CC$-vector spaces
	\begin{align*}
		\mathcal{N}_\pi(X^+) & \rightiso \Hom_{G(\R)_\rho} \left( \pi, C^\infty(X^+_\rho) \right) \\
		\eta & \mapsto \left[ v \mapsto \ell^{-1} \eta(v)|_{X^+(\R)_\rho} \right]
	\end{align*}
\end{proposition}
\begin{proof}
	Since $\eta(v)$ must transform by $\omega_\pi$ under $A_\rho$, we have $\eta(v) \in C^\infty(X^+_\rho) \otimes \omega_\pi|_{A_\rho} \ell$ with respect to $X^+(\R) \simeq X^+(\R)_\rho \times A_\rho$. The bijectivity is then evident.
\end{proof}

\subsection{The \texorpdfstring{$\gamma$}{gamma}-factor}\label{sec:gamma}

\begin{proposition}\label{prop:zeta-uniqueness}
	Consider $\lambda \in \Lambda_{\CC}$ with $\Re(\lambda) \relgeq{X} \kappa$. If $\eta \in \mathcal{N}_\pi(X^+)$ satisfies
	\[ Z_\lambda(\eta, v, \xi) = 0, \quad v \in V_\pi^{K\text{-fini}}, \; \xi \in C_c^\infty(X^+), \]
	then $\eta = 0$.
\end{proposition}
\begin{proof}
	Since we are in the range of convergence, $Z_\lambda(\eta, v, \cdot) = 0$ on $C^\infty_c(X^+)$ implies $\eta(v) |f|^\lambda = 0$, thus $\eta(v) = 0$. Since $V_\pi^{K\text{-fini}}$ is dense in $V_\pi$, it follows that $\eta = 0$.
\end{proof}

\begin{proposition}\label{prop:zeta-restricted}
	Let $\xi \in C^\infty_c(X^+)$. Then $Z_\lambda(\eta, v, \xi)$ is given by $\int_{X^+(\R)} \eta(v) |f|^\lambda \xi$ for any $\lambda \in \Lambda_{\CC}$ off the poles.
\end{proposition}
\begin{proof}
	Evident when $\Re(\lambda) \relgeq{X} \kappa$. The general case follows by meromorphic continuation.
\end{proof}

Before proving the next result, recall that $C^\infty(X^+)$ and $\Schw(X)$ are both smooth as $G(\R)$-representations; the action of $G(\R)$ (resp.\ $\mathfrak{g}$) on them are denoted as $\xi \mapsto g \cdot \xi$ (resp.\ $\xi \mapsto H \cdot \xi$). The $\mathfrak{g}$-action here is derived from the $G(\R)$-action. It differs from the one derived from $\mathfrak{g} \subset U(\mathfrak{g}) \to \mathcal{D}_X$ together with \eqref{eqn:derivation-Schwartz}, because $|\Omega|^{1/2}$ is not $G(\R)$-invariant.

Observe that for smooth $G(\R)$-representations $\pi_1$, $\pi_2$ and a jointly continuous $G(\R)$-invariant bilinear form $B: V_{\pi_1} \times V_{\pi_2} \to \CC$, we have
\begin{equation}\label{eqn:invariant-bilinear-derivative}
	B(\pi_1(H)v_1, v_2) + B(v_1, \pi_2(H)v_2) = 0, \quad H \in \mathfrak{g},\; v_1 \in V_{\pi_1}, \; v_2 \in V_{\pi_2}.
\end{equation}
The argument for \eqref{eqn:invariant-bilinear-derivative} is well-known: simply compute the derivative at $t=0$ of
\[ B\left( \pi_1(\exp(tH)) v_1, \pi_2(\exp(tH)) v_2 \right) = B(v_1, v_2) \quad (t \in \R) \]
using the joint continuity of $B$ and smoothness of $\pi_1$, $\pi_2$.

\begin{lemma}\label{prop:weak-LFE-aux}
	Let $\check{\eta} \in \mathcal{N}_\pi(\check{X}^+)$. Choose a ``denominator'' $L(\check{\eta}, \lambda)$ as in Theorem \ref{prop:meromorphy} and consider the variant $L\check{Z}_\lambda(\check{\eta}, \cdot, \cdot)$ of zeta integrals on $\check{X}$. Fix $\lambda \in \Lambda_{\CC}$ and set $\pi_\lambda := \pi \otimes |\omega|^\lambda$.
	\begin{enumerate}[(i)]
		\item For all $v \in V_\pi^{K\text{-fini}}$, there exists a unique $T_\lambda(v) \in C^\infty(X^+)$ such that
		\[ L\check{Z}_\lambda(\check{\eta}, v, \mathcal{F}\xi) = \int_{X^+(\R)} T_\lambda(v) \xi, \quad \xi \in C^\infty_c(X^+). \]
		\item $v \mapsto T_\lambda(v)$ extends to an element of $\mathcal{N}_{\pi_\lambda}(X^+)$.
	\end{enumerate}
\end{lemma}
\begin{proof}
	Let $v \in V_\pi^{K\text{-fini}}$ $g \in K$ and $H \in \mathfrak{g}$. Since $\mathcal{F}: \Schw(X) \to \Schw(\check{X})$ intertwines smooth $G(\R)$-representations and $L\check{Z}_\lambda(\check{\eta}, \cdot, \cdot)$ is $G(\R)$-invariant and jointly continuous on $\pi_\lambda \times \Schw(X)$ by Theorem \ref{prop:meromorphy} (ii), we have
	\begin{equation}\label{eqn:weak-LFE-aux}\begin{aligned}
		L\check{Z}_\lambda\left( \check{\eta}, v, \mathcal{F} (g \cdot \xi)\right) & = L\check{Z}_\lambda\left( \check{\eta}, v, g \cdot (\mathcal{F} \xi)\right) \\
		& = L\check{Z}_\lambda\left( \check{\eta}, \pi_\lambda(g^{-1}) v, \mathcal{F} \xi \right), \\
		L\check{Z}_\lambda\left( \check{\eta}, v, \mathcal{F} (H \cdot \xi)\right) & = L\check{Z}_\lambda\left( \check{\eta}, v, H \cdot (\mathcal{F} \xi)\right) \\
		& = - L\check{Z}_\lambda\left( \check{\eta}, \pi_\lambda(H) v, \mathcal{F} \xi \right) \quad \because\;\text{\eqref{eqn:invariant-bilinear-derivative}},
	\end{aligned}\end{equation}
	for all $\xi \in \Schw(X)$.

	Since $v$ is finite under $\mathcal{Z}(\mathfrak{g})$ and $K$ with respect to $\pi_\lambda$, the distribution $C_c^\infty(X^+) \ni \xi \mapsto L\check{Z}_\lambda(\check{\eta}, v, \mathcal{F}\xi)$ is also finite under $\mathcal{Z}(\mathfrak{g})$ and $K$ by \eqref{eqn:weak-LFE-aux}. The same holds if we choose $\Omega$ and consider the linear functional $\xi_0 \mapsto L\check{Z}_\lambda(\check{\eta}, v, \mathcal{F}(\xi_0 |\Omega|^{1/2}))$ on $\Schw_0(X)$, since the $G(\R)$-actions on $\xi_0$ and $\xi_0 |\Omega|^{1/2}$ only differ by a character.

	It is then a well-known consequence of the elliptic regularity theorem that our distribution is represented by a unique $T_\lambda(v) \in C^\infty(X^+)$: a detailed explanation can be found in \cite[Proposition 9.7]{Li19}. In fact, $T_\lambda(v)$ is $K$-admissible in the sense of \textit{loc.\ cit.}; see also the Example 2.4 therein.
	
	The $K$-admissibility of the distribution $T_\lambda(v)$, or more generally, of $\mathcal{D}_{X^+_{\CC}}$-module $\mathcal{M}$ it generates, actually implies that $T_\lambda(v)$ is of \emph{moderate growth at infinity}; see \cite[Theorem 9.5]{Li19}.
	
	Now vary $v$. It is clear that $v \mapsto T_\lambda(v)$ is linear, and for all $\xi \in C^\infty_c(X^+)$ we have
	\begin{align*}
		\int_{X^+(\R)} T_\lambda\left( \pi_\lambda(g)v \right) \xi & = L\check{Z}_\lambda \left( \check{\eta}, \pi_\lambda(g)v, \mathcal{F}\xi \right) \stackrel{\eqref{eqn:weak-LFE-aux}}{=} L\check{Z}_\lambda \left( \check{\eta}, v, \mathcal{F}(g^{-1} \cdot \xi) \right) \\
		& = \int T_\lambda(v) (g^{-1} \cdot \xi) = \int \left( g \cdot T_\lambda(v) \right) \xi , \\
		\int_{X^+(\R)} T_\lambda\left( \pi_\lambda(H)v \right) \xi & = L\check{Z}_\lambda \left( \check{\eta}, \pi_\lambda(H)v, \mathcal{F}\xi \right) \stackrel{\eqref{eqn:weak-LFE-aux}}{=} - L\check{Z}_\lambda \left( \check{\eta}, v, \mathcal{F}(H \cdot \xi) \right) \\
		& = - \int T_\lambda(v) (H \cdot \xi) \stackrel{\eqref{eqn:invariant-bilinear-derivative}}{=} \int \left( H \cdot T_\lambda(v) \right) \xi,
	\end{align*}
	where we used the fact that $\int: C^\infty(X^+) \times C^\infty_c(X^+) \to \CC$ is $G(\R)$-invariant and jointly continuous. Indeed, invariance follows by change of variables, whilst the joint continuity is easily checked by restricting to $C^\infty(X^+) \times C^\infty_\Omega(X^+)$ and recalling the topologies from \S\S\ref{sec:Fourier}---\ref{sec:coeff-rep}, where $\Omega \subset X^+(\R)$ is any compact subset.
	
	As $\xi$ is arbitrary, we deduce
	\[ T_\lambda\left( \pi_\lambda(g) v \right) = g \cdot T_\lambda(v), \quad T_\lambda\left( \pi_\lambda(H) v \right) = H \cdot T_\lambda(v). \]
	
	Summing up, $T_\lambda: V_{\pi_\lambda}^{K\text{-fini}} \to C^\infty(X^+)^{K\text{-fini}}$ is a map of $(\mathfrak{g}, K)$-modules. We claim that $T_\lambda$ extends to an element of $\mathcal{N}_{\pi_\lambda}(X^+)$. Indeed, this would follow from \cite[Example 11.1 (b) and Proposition 11.2]{BK14} provided that $T_\lambda(v)$ is of \emph{moderate growth} on $X^+(\R)$ for every $v \in V_\pi^{K\text{-fini}}$. To reconcile the aforementioned moderate growth at infinity in \cite{Li19} with that in \cite{BK14}, see the proof of Proposition \ref{prop:soft-bound}.
\end{proof}

\begin{proposition}[Weak functional equation]\label{prop:weak-LFE}
	Let $\pi$ be an SAF representation of $G(\R)$ with central character. There exists a unique meromorphic map $\gamma(\pi, \lambda): \mathcal{N}_\pi(\check{X}^+) \to \mathcal{N}_\pi(X^+)$ (i.e.\ its matrix entries are meromorphic in $\lambda \in \Lambda_{\CC}$), such that for all $\check{\eta} \in \mathcal{N}_\pi(\check{X}^+)$ and all $v \in V_\pi$, we have
	\[ \check{Z}_\lambda\left( \check{\eta}, v, \mathcal{F}\xi \right) = Z_\lambda\left( \gamma(\lambda, \pi)(\check{\eta}), v, \xi \right), \quad \xi \in C^\infty_c(X^+), \]
	for all $\lambda \in \Lambda_{\CC}$ off the poles. Moreover:
	\begin{enumerate}[(i)]
		\item $\gamma(\pi, \lambda)$ is unique: if $\gamma_1(\pi, \lambda)$, $\gamma_2(\pi, \lambda)$ satisfy
		\[ Z_\lambda\left( \gamma_1(\lambda, \pi)(\check{\eta}), v, \xi \right) = Z_\lambda\left( \gamma_2(\lambda, \pi)(\check{\eta}), v, \xi \right) \]
		for all $\check{\eta}, v, \xi$ and all $\lambda$ in an open subset $U \neq \emptyset$ in $\Lambda_{\CC}$, then $\gamma_1(\pi, \lambda) = \gamma_2(\pi, \lambda)$ as meromorphic families in $\lambda$;
		\item if $L(\check{\eta}, \lambda)$ is as in Theorem \ref{prop:meromorphy}, then $L(\check{\eta}, \lambda) \gamma(\pi, \lambda)$ is holomorphic in $\lambda$;
		\item $\left( \gamma(\pi, \lambda + \mu)(\check{\eta}) \right)_\lambda = \gamma(\pi \otimes |\omega|^\lambda, \mu)(\check{\eta}_\lambda)$ for all $\mu, \lambda, \check{\eta}$.
	\end{enumerate}
\end{proposition}
\begin{proof}
	Write $\pi_\lambda := \pi \otimes |\omega|^\lambda$ as before. Let $\check{\eta} \in \mathcal{N}_\pi(\check{X}^+)$ and choose a ``denominator'' $L(\check{\eta}, \lambda)$ as in Lemma \ref{prop:weak-LFE-aux} to obtain the family $T_\lambda \in \mathcal{N}_{\pi_\lambda}(X^+)$ in $\lambda \in \Lambda_{\CC}$. Define
	\[ L\gamma(\lambda, \pi)(\check{\eta}) := T_\lambda(\cdot) |f|^{-\lambda} \; \in \mathcal{N}_\pi(X^+). \]
	We contend that $L\gamma(\lambda, \pi)(\check{\eta})$ is a holomorphic family inside $\mathcal{N}_\pi(X^+)$. Lemma \ref{prop:weak-LFE-aux} implies
	\[ L\check{Z}_\lambda(\check{\eta}, v, \mathcal{F}\xi) = \int_{X^+(\R)} L\gamma(\lambda, \pi)(\check{\eta})(v) \cdot |f|^\lambda \xi, \quad v \in V_\pi^{K\text{-fini}}. \]
	The left hand side being holomorphic in $\lambda$ (while $v, \xi$ are kept fixed), our strategy is to repeat the arguments for Lemma \ref{prop:holomorphy-test} to prove our claim. The problem, however, is the presence of $|f|^\lambda$ in the integrand. The workaround is to use the decomposition $X^+(\R) \simeq X^+(\R)_\rho \times A_\rho$ in Proposition \ref{prop:rho-decomposition}. Let
	\begin{compactitem}
		\item $\xi = \xi_1 \otimes \xi_2$ with $\xi_1 \in C^\infty_c(X^+_\rho)$ and $\xi_2 \in C^\infty_c(A_\rho)$;
		\item $\ell$: an invariant, nonzero half-density on $A_\rho$;
		\item $L\gamma(\gamma, \pi)(\check{\eta})(v) = U_\lambda(v) \otimes \omega_\pi \ell$, where $U_\lambda \in \Hom_{G(\R)_\rho}\left(\pi,  C^\infty(X^+_\rho)\right)$.
	\end{compactitem}
	Such a decomposition of $L\gamma(\gamma, \pi)(\check{\eta})(v)$ exists and is unique (Proposition \ref{prop:N-restriction}). We have
	\[ \int_{X^+(\R)} L\gamma(\lambda, \pi)(\check{\eta})(v) \cdot |f|^\lambda \xi = \int_{X^+(\R)_\rho} U_\lambda(v) \xi_1 \cdot \int_{A_\rho} \omega_\pi |\omega|^\lambda \ell \xi_2. \]
	
	The integral $\int_{A_\rho}$ is holomorphic in $\lambda$. For every given $\lambda_\circ \in \Lambda_{\CC}$, we may choose $\xi_2$ such that $\int_{A_\rho} \omega_\pi |\omega|^{\lambda_\circ} \xi_2 \neq 0$, and the non-vanishing propagates to some neighborhood $\mathcal{U}$ of $\lambda_\circ$.
	
	It follows that $\int_{X^+(\R)_\rho} U_\lambda(v) \xi_1$ is holomorphic in $\lambda$ over $\mathcal{U}$, for all $\xi_1 \in C^\infty_c(X^+_\rho)$ and $v \in V_\pi^{K\text{-fini}}$. Hence the arguments for (ii) $\implies$ (i) in Lemma \ref{prop:holomorphy-test} show that $U_\lambda$ is a holomorphic family inside $\Hom_{G(\R)_\rho}(\pi, C^\infty(X^+_\rho))$. Our claim on the holomorphy of $L\gamma(\lambda, \pi)(\check{\eta})$ inside $\mathcal{N}_\pi$ follows from Proposition \ref{prop:N-restriction}.

	Next, consider the meromorphic family in $\lambda \in \Lambda_{\CC}$:
	\[ \gamma(\lambda, \pi)(\check{\eta}) := \frac{L\gamma(\lambda, \pi)(\check{\eta})}{L(\check{\eta}, \lambda)}, \quad \check{\eta} \in \mathcal{N}_\pi(\check{X}^+). \]
	It satisfies $\check{Z}_\lambda \left(\check{\eta}, v, \mathcal{F}\xi \right) = Z_\lambda\left( \gamma(\lambda, \pi)(\check{\eta}), v, \xi \right)$ for all $\xi \in C^\infty_c(X^+)$ and $v \in V_\pi^{K\text{-fini}}$. The equality extends to all $v \in V_\pi$ by continuity.

	Consider the assertion (i). We may assume that $U$ is disjoint from the singularities of $Z_\lambda$. Proposition \ref{prop:zeta-restricted} implies $\gamma_1(\lambda, \pi) = \gamma_2(\lambda, \pi)$ for all $\lambda \in U$, hence determines $\gamma(\lambda, \pi)$ as a meromorphic family in $\lambda \in \Lambda_{\CC}$.

	Assertion (ii) follows from the construction of $\gamma(\pi, \lambda)$. As for (iii), notice that
	\begin{align*}
		\check{Z}_{\lambda + \mu}\left(\check{\eta}, v, \mathcal{F}\xi \right) & = \check{Z}_\mu\left( \check{\eta}_\lambda, v, \mathcal{F}\xi \right) \quad (\because\;\text{Corollary \ref{prop:zeta-translation}}) \\
		& = Z_\mu \left( \gamma(\mu, \pi_\lambda)(\check{\eta}_\lambda), v, \xi \right)
	\end{align*}
	for all $\xi \in C^\infty_c(X^+)$. On the other hand,
	\begin{align*}
		\check{Z}_{\lambda + \mu}\left(\check{\eta}, v, \mathcal{F}\xi \right) & = Z_{\lambda + \mu}\left( \gamma(\pi, \lambda + \mu)(\check{\eta}), v, \xi \right) \\
		& = Z_\mu\left( \left( \gamma(\pi, \lambda + \mu)(\check{\eta}) \right)_\lambda, v, \xi \right) \quad (\because\;\text{Corollary \ref{prop:zeta-translation}}).
	\end{align*}
	When $\Re(\mu) \relgeq{X} \kappa$ and $\mu$ lies off the poles, (iii) follows by applying Proposition \ref{prop:zeta-uniqueness}. The general case of (iii) follows by meromorphic continuation.
\end{proof}

\begin{remark}
	The uniqueness of $\gamma(\pi, \lambda)$ in the weak functional equation has been established in \cite[\S 4.5]{LiLNM} by the same reasoning. The proof above can also be applied in the general setting in \textit{loc.\ cit.} to furnish a $\gamma$-factor together with a weak functional equation, provided that the axioms thereof are satisfied. Since the framework in \textit{loc.\ cit.} is largely conjectural, we confine ourselves to the case of prehomogeneous vector spaces here.
\end{remark}

We close this subsection by the compatibility between $\gamma$-factors and intertwining operators.
\begin{proposition}\label{prop:gamma-vs-intertwining}
	Let $\varphi: \pi \to \sigma$ be a morphism between SAF representations of $G(\R)$. Define $\varphi^*: \mathcal{N}_\sigma(X^+) \to \mathcal{N}_\pi(X^+)$ by $\eta \mapsto \eta \circ \varphi$, and similarly for $\check{X}^+$.
	\begin{enumerate}[(i)]
		\item For all $\eta \in \mathcal{N}_\sigma(X^+)$, $v \in V_\pi$ and $\xi \in \Schw(X)$, we have $Z_\lambda\left( \eta, \varphi(v), \xi \right) = Z_\lambda\left( \varphi^* \eta, v, \xi \right)$.
		\item Suppose that $\pi, \sigma$ have central characters. Then $\gamma(\lambda, \pi) \circ \varphi^* = \varphi^* \circ \gamma(\lambda, \sigma)$.
	\end{enumerate}
\end{proposition}
\begin{proof}
	Assertion (i) is clear in the range of convergence; the general case follows by meromorphic continuation. As for (ii), it suffices to observe that by (i),
	\begin{multline*}
		Z_\lambda\left( \gamma(\lambda, \pi) \varphi^* \check{\eta}, v, \xi \right) = \check{Z}_\lambda\left( \varphi^* \check{\eta}, v, \mathcal{F}\xi \right) = \check{Z}_\lambda\left( \check{\eta}, \varphi(v), \mathcal{F}\xi \right) \\
		= Z_\lambda \left( \gamma(\lambda, \sigma)\check{\eta}, \varphi(v), \xi \right) = Z_\lambda\left( \varphi^* \gamma(\lambda, \sigma) \check{\eta}, v, \xi \right)
	\end{multline*}
	for all $\check{\eta} \in \mathcal{N}_\sigma(\check{X}^+)$, $v \in V_\pi$ and $\xi \in \Schw(X)$. Now apply Proposition \ref{prop:zeta-uniqueness}.
\end{proof}

\subsection{Consequences of the weak functional equation}
Fix an SAF representation $\pi$ of $G(\R)$ with central character. With the notations of Proposition \ref{prop:weak-LFE}, we define
\begin{equation}\label{eqn:Delta}
	\Delta_\lambda(\check{\eta}, v, \xi) := \check{Z}_\lambda(\check{\eta}, v, \mathcal{F}\xi) - Z_\lambda\left( \gamma(\pi, \lambda)(\check{\eta}), v, \xi \right)
\end{equation}
for all $\check{\eta} \in \mathcal{N}_\pi(\check{X}^+)$, $v \in V_\pi$ and $\xi \in \Schw(X)$. Note that $\Delta_\lambda(\check{\eta}, v, \cdot)$ is a meromorphic family of tempered distribution on $X$. Our Theorem \ref{prop:LFE} amounts to $\Delta_\lambda(\check{\eta}, v, \xi) = 0$, and it suffices to check this for $\lambda$ in any given open subset $U \neq \emptyset$ of $\Lambda_{\CC}$.

\begin{lemma}\label{prop:Delta-h}
	Let $U \subset \Lambda_{\CC}$ be a nonempty open subset such that
	\begin{compactitem}
		\item the closure of $U$ is compact,
		\item $U$ is disjoint from the singularities of $Z_\lambda$, $\check{Z}_\lambda$ and $\gamma(\pi, \lambda)$.
	\end{compactitem}
	For any $h \in \R[X]$ such that $\partial X = \{ x: h(x) = 0 \}$, there exists $M \in \Z_{\geq 0}$ such that
	\[ \Delta_\lambda \left( \check{\eta}, v, h^M \xi \right) = 0, \quad \lambda \in U, \]
	for all $\check{\eta}$, $v$ and $\xi$.
\end{lemma}
\begin{proof}
	Pick $c, \check{c} \in \Lambda_{\R}$ such that $\Re(\lambda) \relgeq{X} c$ and $\Re(\lambda) \relgeq{\check{X}} \check{c}$ for all $\lambda \in U$. By Propositions \ref{prop:order-Fourier} and \ref{prop:order-estimate}, there exists $(a, b) \in \Z^2_{\geq 0}$ such that $\xi \mapsto \Delta_\lambda \left( \check{\eta}, v, \xi \right)$ has order $\leq (a, b)$ whenever $\lambda \in U$. Furthermore, $\Delta_\lambda \left( \check{\eta}, v, \xi \right) = 0$ for $\xi \in C^\infty_c(X^+)$. The assertion is then well-known; see for instance \cite[Proposition 3.15]{Ki03}.
\end{proof}

Let $h \in \R[X]$, $\check{h} \in \R[\check{X}]$ be a pair of relative invariants as in Corollary \ref{prop:rel-invariant-pair}, with eigencharacters $-\theta$ and $\theta$ respectively. Upon multiplying $\check{h}, h$ by some positive real numbers, we may and do assume that
\begin{equation}\label{eqn:check-h-theta}
	h(x) = |f|^{-\theta}(x), \quad \check{h}(y) = |\check{f}|^\theta (y), \quad x \in X^+(\R), \; y \in \check{X}^+(\R).
\end{equation}

Plug the choice above of $h$ into the setting of Lemma \ref{prop:Delta-h}, and take $U$ and $M$ as in that Lemma. Take a $\check{\kappa} \in \Lambda_{\R}$ associated with $\pi$ and $(G, \check{\rho}, \check{X})$ as in Theorem \ref{prop:convergence}. Observe that for all $\check{\eta}$, $v$ and $\xi \in \Schw(X)$,
\begin{align*}
	\check{Z}_\lambda\left( \check{\eta}, v, \mathcal{F}(h^M \xi) \right) & = c(\psi)^{M \deg h} \check{Z}_\lambda\left( \check{\eta}, v, h^M (\mathcal{F}\xi) \right) \quad \because \text{Lemma \ref{prop:F-commute}} \\
	& = c(\psi)^{M \deg h} \int_{\check{X}^+(\R)} \check{\eta}(v) |\check{f}|^\lambda \cdot h^M (\mathcal{F}\xi) \quad \text{assuming}\; \Re(\lambda) \relgeq{\check{X}} \check{\kappa} \\
	& = (-c(\psi))^{M \deg h} \int_{\check{X}^+(\R)} h^M \left( \check{\eta}(v) |\check{f}|^\lambda \right) \mathcal{F}\xi \quad \\
	& \quad \because\text{integration by parts on $X(\R)$ and \eqref{eqn:derivation-Schwartz}} \\
	& = (-c(\psi))^{M \deg h} \int_{\check{X}^+(\R)} C_\lambda\left( \check{h}^M \otimes h^M \right)(\check{\eta}(v)) \cdot |\check{f}|^{\lambda - M\theta} \cdot \mathcal{F}\xi \quad \because\text{\eqref{eqn:check-h-theta}} \\
	& = (-c(\psi))^{M \deg h} \check{Z}_{\lambda - M\theta} \left( C_\lambda\left( \check{h}^M \otimes h^M \right)\check{\eta}, v, \mathcal{F}\xi \right).
\end{align*}
The first and the last terms are both meromorphic in $\lambda$ when $\check{\eta}$ is fixed (Proposition \ref{prop:Capelli-holomorphy}), hence the equality extends to all $\lambda$.

On the other hand, Corollary \ref{prop:zeta-translation} and \eqref{eqn:check-h-theta} imply
\[ Z_\lambda\left( \gamma(\pi, \lambda)\check{\eta}, v, h^M \xi\right) = Z_{\lambda - M\theta} \left( \gamma(\pi, \lambda)\check{\eta}, v, \xi \right). \]

Let the open subset $U$ be as in Lemma \ref{prop:Delta-h}. For all $\lambda \in U$ we arrive at
\begin{equation}\label{eqn:Delta-h-new}
	(-c(\psi))^{M \deg h} \check{Z}_{\lambda - M\theta} \left( C_\lambda\left( \check{h}^M \otimes h^M \right)\check{\eta}, v, \mathcal{F}\xi \right) = Z_{\lambda - M\theta} \left( \gamma(\pi, \lambda)\check{\eta}, v, \xi \right).
\end{equation}

For every $\check{\eta} \in \mathcal{N}_\pi(\check{X}^+)$, define
\begin{equation}\label{eqn:lambda-eta}
	{}^\lambda \check{\eta} := (-c(\psi))^{M \deg h} C_\lambda\left( \check{h}^M \otimes h^M \right) \check{\eta}.
\end{equation}
It is linear in $\check{\eta}$ and gives a holomorphic family (in $\lambda$) inside $\mathcal{N}_\pi(\check{X}^+)$ by Proposition \ref{prop:Capelli-holomorphy}.

\begin{lemma}\label{prop:consequence-weak-LFE}
	Take $h \in \R[X]$, $\check{h} \in \R[\check{X}]$, $\theta \in \Lambda_{\Z}$, $U \subset \Lambda_{\CC}$ and $M \in \Z_{\geq 0}$ as in the recipe above. Let
	\[ U' := U \smallsetminus \text{singularities of } Z_{\lambda - M\theta}, \check{Z}_{\lambda - M\theta}, \gamma(\pi, \lambda - M\theta) \]
	so that $U'$ is open dense in $U$. Then $\Delta_{\lambda - M\theta}\left( {}^\lambda \check{\eta}, v, \xi \right) = 0$ for $\lambda \in U'$, that is,
	\[ \check{Z}_{\lambda - M\theta} \left( {}^\lambda \check{\eta}, v, \mathcal{F}\xi \right) = Z_{\lambda - M\theta} \left( \gamma\left( \pi, \lambda - M\theta \right)( {}^\lambda \check{\eta} ), v, \xi \right), \quad \lambda \in U' , \]
	where $\check{\eta} \in \mathcal{N}_\pi(\check{X}^+)$, $v \in V_\pi$ and $\xi \in \Schw(X)$ are arbitrary.
\end{lemma}
\begin{proof}
	In view of \eqref{eqn:Delta-h-new}, \eqref{eqn:lambda-eta} and Proposition \ref{prop:weak-LFE} (i), we deduce
	\[ \gamma\left( \pi, \lambda - M\theta \right)( {}^\lambda \check{\eta} ) = \gamma\left( \pi, \lambda \right)( \check{\eta} ), \quad \check{\eta} \in \mathcal{N}_\pi(\check{X}^+) \]
	as meromorphic families in $\lambda$. Plugging this back into \eqref{eqn:Delta-h-new} yields the asserted equality.
\end{proof}

\subsection{Proof of functional equation}
Fix an SAF representation $\pi$ of $G(\R)$ with central character. Let $h \in \R[X]$, $\check{h} \in \R[\check{X}]$ be a pair of relative invariants as in Corollary \ref{prop:rel-invariant-pair}, satisfying \eqref{eqn:check-h-theta}. Consider the linear map
\[\begin{tikzcd}[row sep=tiny]
	\Phi_\lambda : \mathcal{N}_\pi(\check{X}^+) \arrow[r] & \mathcal{N}_\pi(\check{X}^+) \\
	\eta \arrow[mapsto, r] & C_\lambda\left( \check{h}^M \otimes h^M \right) \circ \eta.
\end{tikzcd}\]
It is holomorphic in $\lambda \in \Lambda_{\CC}$ (i.e.\ its matrix entries are all holomorphic if we fix a basis) by Proposition \ref{prop:Capelli-holomorphy}.

\begin{lemma}\label{prop:det-nonvanishing}
	The holomorphic function $\lambda \mapsto \det \Phi_\lambda$ on $\Lambda_{\CC}$ is not identically zero.
\end{lemma}
\begin{proof}
	Observe that the commutative $\CC$-algebra
	\[ \mathcal{D}(\check{X}^+_{\CC})^{G_{\CC}} = \mathcal{D}(\check{X}^+)^G \dotimes{\R} \CC \]
	acts on $\mathcal{N}_\pi(\check{X}^+)$ by $\check{\eta} \mapsto D_* \check{\eta} := D \circ \check{\eta}$, where $D \in \mathcal{D}(\check{X}^+_{\CC})^{G_{\CC}}$. Using Theorem \ref{prop:Knop-HC} and Remark \ref{rem:inf-char}, $\mathcal{N}_\pi(\check{X}^+)$ decomposes into joint generalized eigenspaces
	\begin{align*}
		\mathcal{N}_\pi(\check{X}^+) & = \bigoplus_\chi \mathcal{N}_\chi, \\
		\mathcal{N}_\chi & := \left\{ \check{\eta} \in \mathcal{N}_\pi(\check{X}^+) : \text{has infinitesimal character } \chi \right\}
	\end{align*}
	under $\mathcal{D}(\check{X}^+_{\CC})^{G_{\CC}} = \mathcal{D}(\check{X}_{\CC})^{G_{\CC}}$ (Lemma \ref{prop:auto-regularity}), where $\chi$ ranges over $\left( \rho + \mathfrak{a}_\mathbf{X}^* \right) /\!/ W_X$. It suffices to show that $\det \left( \Phi_\lambda \middle| \mathcal{N}_\chi \right)$ is not identically zero, for each $\chi$.

	Take a non-degenerate relative invariant $\check{g} \in \R[\check{X}]$ such that $\check{g} \geq 0$ (for example $\check{g} = \check{h}$). Multiplying by some positive constant, we may assume $\check{g} = |\check{f}|^\mu$ on $\check{X}^+(\R)$ where $\mu \in \Lambda_{\Z}$ is the eigencharacter of $\check{g}$. Set
	\[ D := C \left( \check{h}^M \otimes h^M \right). \]

	Define $D_{\check{g}, s} \in \mathcal{D}(\check{X})^G$ as in Proposition \ref{prop:HC-twist}, where $s \in \Z$. We claim that for all but finitely many $s$, we have
	\begin{equation}\label{eqn:nonvanishing}
		\det \left( (D_{\check{g}, s})_*: \mathcal{N}_\chi \to \mathcal{N}_\chi \right) \neq 0,
	\end{equation}
	This will conclude the proof since Remark \ref{rem:analytic-twist} says that $D_{\check{g}, s} = C_{s\mu}\left( \check{h}^M \otimes h^M \right)$.
	
	To show \eqref{eqn:nonvanishing}, we deduce from Proposition \ref{prop:HC-twist} that $(D_{\check{g}, s})_* \in \End_{\CC}(\mathcal{N}_\chi)$ has the generalized eigenvalue
	\begin{align*}
		\mathrm{HC}(D_{\check{g}, s})(\chi) & = \mathrm{HC}(D)(\chi - s\mu) = c_D(\chi - \rho - s\mu)  \\
		& = s^{\deg c_D} \cdot c_D^{\mathrm{top}}(-\mu) + (\text{lower terms in $s$}).
	\end{align*}
	The top homogeneous component $c_D^{\mathrm{top}}$ satisfies $c_D^{\mathrm{top}}(-\mu) \neq 0$ by Proposition \ref{prop:top-nonvanishing}, because $\check{g} \in \R[\check{X}]$ is non-degenerate. This establishes \eqref{eqn:nonvanishing}.
\end{proof}

\begin{proof}[Proof of Theorem \ref{prop:LFE}]
	Take Lemma \ref{prop:consequence-weak-LFE} as our foothold. Lemma \ref{prop:det-nonvanishing} implies that $\eta \mapsto {}^\lambda \eta$ is a linear automorphism of $\mathcal{N}_\pi(\check{X}^+)$ on an open dense subset $U'' \subset U'$. Hence for any given $\lambda \in U''$,
	\[ \Delta_{\lambda - M\theta}(\check{\eta}, v, \xi) = 0 \]
	holds for all $\check{\eta} \in \mathcal{N}_\pi(\check{X}^+)$, $v \in V_\pi$ and $\xi \in \Schw(X)$; recall that $\Delta$ is defined in \eqref{eqn:Delta}.
	
	Since $U'' - M\theta$ is open nonempty in $\Lambda_{\CC}$ and $\lambda \mapsto \Delta_\lambda(\check{\eta}, v, \xi)$ is known to be meromorphic, the equality extends to the whole $\Lambda_{\CC}$. This proves the functional equation in Theorem \ref{prop:LFE}; the remaining assertions about $\gamma$-factors are already established in Proposition \ref{prop:weak-LFE}.
\end{proof}

% Below: for Biblatex...
%\printbibliography[heading=bibintoc]

% Below: for bibtex - preferred on arXiv
\bibliographystyle{abbrv}
\bibliography{Capelli}

\vspace{1em}
\begin{flushleft}
	Wen-Wei \textsc{Li} \\
	E-mail address: \href{mailto:wwli@bicmr.pku.edu.cn}{\texttt{wwli@bicmr.pku.edu.cn}} \\
	School of Mathematical Sciences, Peking University \\
	No.\ 5 Yiheyuan Road, Beijing 100871, People's Republic of China.
\end{flushleft}

\end{document}